%% file: Parareal_for_CH.tex
\documentclass[reqno]{amsart}
\usepackage{hyperref}

\DeclareMathOperator{\diag}{diag}
\usepackage{amsmath}
\usepackage{mathabx}
\usepackage{subcaption}
\usepackage{graphicx}
\usepackage{mathtools}

\begin{document}

\title[\hfil  Parareal for the CH Equation]
{Linear and Nonlinear Parareal Methods for the Cahn-Hilliard Equation}
 
\author[G. Garai. B. C. Mandal \hfilneg]
{Gobinda Garai, Bankim C. Mandal}

\address{Gobinda Garai \newline
School of Basic Sciences,
Indian Institute of Technology Bhubaneswar, India}
\email{gg14@iitbbs.ac.in}

\address{Bankim C. Mandal \newline
School of Basic Sciences,
Indian Institute of Technology Bhubaneswar, India}
\email{bmandal@iitbbs.ac.in} 

\thanks{Submitted.}

\begin{abstract}
In this paper, we propose, analyze and implement efficient time parallel methods for the Cahn-Hilliard (CH) equation. It is of great importance to develop efficient numerical methods for the CH equation, given the range of applicability of the CH equation has. The CH
equation generally needs to be simulated for a very long time to get the solution of phase coarsening stage. Therefore it is desirable to accelerate the computation
using parallel method in time. We present linear and nonlinear Parareal methods for the CH equation depending on the choice of fine approximation. We illustrate our results by numerical experiments.
\end{abstract}

\maketitle
\numberwithin{equation}{section}
\newtheorem{theorem}{Theorem}[section]
\newtheorem{lemma}[theorem]{Lemma}
\newtheorem{definition}[theorem]{Definition}
\newtheorem{proposition}[theorem]{Proposition}
\newtheorem{remark}[theorem]{Remark}
\allowdisplaybreaks

\section{Introduction} \label{intro}
We are interested in designing time parallel algorithms for the Cahn-Hilliard equation 
\begin{equation}\label{CH}
\begin{cases}
\frac{\partial u}{\partial t} +\epsilon^2\Delta^2 u= \Delta f(u), & (x,t)\in\Omega\times(0,T],\\
u=0=\partial_{\nu}(\nabla u), & (x,t)\in\partial\Omega\times(0,T],\\
u(x,0)=u^0, & x\in \Omega, 
\end{cases} 
\end{equation}
where $\nu$ is the outward unit normal to $\partial\Omega$.
The CH equation has been suggested as a prototype to represent the evolution of a binary melted alloy below the critical temperature in \cite{Cahn,Hilliard}. The CH equation \eqref{CH} also arises from the Ginzburg-Landau energy functional:
\begin{equation}\label{energy}
\mathcal {E}(u):=\int_{\Omega}\left(F(u)+\frac{\epsilon^2}{2}\vert\nabla u\vert^2\right)d\boldsymbol{x},
\end{equation} 
by considering \eqref{CH} as a gradient flow $\frac{\partial u}{\partial t}=\Delta \frac{\delta \mathcal{E}}{\delta u}$, where $\frac{\delta \mathcal{E}}{\delta u}$ is the first variation of energy, $F(u) = 0.25(u^2 - 1)^2$ with $F'(u)=f(u)$, and $\epsilon (0<\epsilon\ll 1)$ is the thickness of the interface.  
The solution of \eqref{CH} involves two different dynamics, one is phase separation which is quick in time, and another is phase coarsening which is slow in time. The fine-scale phase regions are formed during the early stage of the dynamics of width $\epsilon$. Whereas during the phase coarsening stage, the solution tends to an equilibrium state which minimizes the system energy in \eqref{energy}.
By differentiating the energy functional $\mathcal{E}(u)$ and total mass $\int_{\Omega}u$ with respect to time $t$, we get 
\begin{equation}\label{energy minimization}
\frac{d}{dt}\mathcal{E}(u)\leq 0, \mbox{\hspace{1cm}} \frac{d}{dt}\int_{\Omega}u = 0.
\end{equation}
So the CH equation describes energy minimization and the total mass conservation while the system evolves. 

The existence of the solution of the CH equation \eqref{CH} can be seen form \cite{DuNicolaides} and also results for other variants of the CH equation are shown in \cite{ElliottZheng,liuzhao}. Various research have been done in finding numerical scheme for the CH equation to approximate the solution with either Dirichlet \cite{DuNicolaides,David} or Neumann boundary conditions \cite{elliott1987numerical,furihata2001stable,vollmayr2003fast,shin2011conservative} and references therein. Recently a new approach to approximate the solution of the CH equation has been proposed in \cite{yang2016linear, yang2017numerical} based on quadratization of the energy $\mathcal{E}(u)$ of the CH equation. A modification on energy quadratization approach yields a new method known as scalar auxiliary variable \cite{shen2018scalar}. A review on numerical treatment of the CH equation can be found in \cite{lee2014physical}. The possible application of CH equation as a model are: image inpainting \cite{EsedoAndrea}, tumour growth simulation \cite{Wise}, population dynamics \cite{cohen1981generalized}, dendritic growth \cite{kim1999universal}, planet formation \cite{tremaine2003origin}, etc.

The above described works are all in time stepping fashion for advancement of evolution of the CH equation. Therefore to get a solution of CH equation need to be solved sequentially over long time for capturing the long term behaviour of the CH equation, specially the phase coarsening stage.   
Consequently, it is of great importance to accelerate the simulation using parallel computation, which can be fulfilled by time parallel techniques. 
In last few decades there is a lot of efforts on formulating various type of time parallel techniques, for an overview see \cite{gander50year}. 
To speed-up the computation we construct the Parareal methods for the CH equation \eqref{CH}. The Parareal method \cite{lions2001parareal} is a well known iterative time parallel method, that can also be viewed as multiple shooting method or time-multigrid method; see \cite{gander2007analysis}. The method rely on computing fine and coarse resolution and eventually converge to fine resolution. It has been successfully applied to: fluid-structure interaction in \cite{farhat2003time}, Navier-Stokes equation in \cite{fischer2005parareal}, molecular-dynamics in \cite{baffico2002parallel}. The main objective of this work is to adapt the  Parareal algorithm for the CH equation \eqref{CH} and study the convergence behaviour. 

The rest of this paper is arranged as follows. We introduce in Section \ref{Section1} the time parallel algorithm for equation \eqref{CH}. In section \ref{Section2} we present the stability and convergence results. To illustrate our analysis, the accuracy and robustness of the proposed formulation, we show numerical results in Section \ref{Section3}.

\section{Parareal Method}\label{Section1}
To solve the following system of ODEs  
\begin{equation}\label{ode}
      \frac{du}{dt}=f(t, u),\; u(0)=u^0,\;  t\in(0,T],  
\end{equation}
Lions et al. proposed the Parareal algorithm in \cite{lions2001parareal}, where $f:\mathbb{R}^{+} \times \mathbb{R}^d\rightarrow\mathbb{R}^d$ is Lipschitz.
The method constitutes of the following strategy: first a non-overlapping decomposition of time domain $(0, T]$ into $N$ smaller subintervals of uniform size, i.e., $(0, T]=\cup_{n=1}^{N}[T_{n-1}, T_n]$ with $T_{n}- T_{n-1}=\Delta T=T/N$ is considered,  secondly each time slice $[T_{n-1}, T_n]$ is divided into $J$ smaller time slices with $\Delta t=\Delta T/J$, then a fine propagator $\mathcal{F}$ which is expensive but accurate, and a coarse propagator $\mathcal{G}$ which is cheap but may be inaccurate are assigned to compute the solution in fine grid and coarse grid respectively. Then the Parareal algorithm for \eqref{ode} starts with the initial approximation ${U}_n^0$ at $T_n$'s, obtained by the coarse operator $\mathcal{G}$ and solve the following prediction-correction scheme for $k=0, 1,...$
\begin{equation}\label{parareal}
    \begin{aligned}
      {U}_0^{k+1}& =u^0,\\
      {U}_{n+1}^{k+1}& =\mathcal{G}(T_{n+1}, T_n, {U}_n^{k+1})+\mathcal{F}(T_{n+1}, T_n, {U}_n^k)-\mathcal{G}(T_{n+1}, T_n, {U}_n^k),
    \end{aligned}       
\end{equation}
where operator $\mathcal{S}(T_{n+1}, T_n, {U}_n^{k})$ provides solution at $T_{n+1}$ by taking the initial solution $U_n^k$ at $T_n$ for $\mathcal{S}=\mathcal{F} \text{or}\; \mathcal{G}$. At current iteration ${U}_n^{k}$'s are known, hence one computes $\mathcal{F}(T_{n+1}, T_n, {U}_n^k)$ in parallel using $N$ processor.
The Parareal solution converges towards the fine resolution in finite steps. To get a practical parallel algorithm we should have $k\ll N$.

\subsection{Discretization and Formulation}
To formulate the Parareal method for the CH equation \eqref{CH} we first look into possible discretization of \eqref{CH} in both spatial and temporal variables.
Since the non-increasing of the total energy and mass conservation property \eqref{energy minimization} are essential features of the CH equation \eqref{CH},  they are expected to be preserved for long time simulation under any proposed numerical scheme as well. To deal with this, Eyre proposed an unconditionally gradient stable scheme in \cite{David,Eyre}. The idea is to split the homogeneous free energy $F(u)=\underbrace{\frac{u^4}{4} +1 }_\text{convex part}+ \underbrace{\frac{-u^2}{2}}_\text{concave part}$ into a sum of a convex and a concave term, and then treat the convex term implicitly and the concave term explicitly to obtain a nonlinear approximation for \eqref{CH} in 1D as:
\begin{equation}\label{approx1}
u_j^{n+1} - u_j^{n} =  \Delta t D_h (u_j^{n+1})^3  -\Delta t D_h u_j^{n} - \epsilon^2\Delta t D_h^2 u_j^{n+1},
\end{equation}
where $\Delta t$ is the time step and $D_h$ is the discrete Laplacian and the scheme is $O(\Delta t) + O(\Delta x^2)$ accurate \cite{David,Eyre}. The scheme \eqref{approx1} is unconditionally gradient stable, means the discrete energy is non-increasing for every time step $\Delta t$. To get a linear approximation of \eqref{CH}, the term $(u_j^{n+1})^3$ in \eqref{approx1} is rewritten as $(u_j^{n})^2 u_j^{n+1}$, which leads to the following linear approximation 
\begin{equation}\label{approx2}
u_j^{n+1} - u_j^{n} =  \Delta t D_h (u_j^{n})^2 u_j^{n+1} -\Delta t D_h u_j^{n}- \epsilon^2\Delta t D_h^2 u_j^{n+1}.
\end{equation}
This is also an unconditionally gradient stable scheme and has the same accuracy as the previous nonlinear scheme \eqref{approx1} \cite{David}. Another convex-concave splitting of $F(u)$ is $F(u)=\underbrace{u^2+\frac{1}{4}}_\text{convex part}+ \underbrace{\frac{u^4}{4}-\frac{3u^2}{2}}_\text{concave part},$  and by treating the convex part implicitly and concave part explicitly one obtains the following unconditionally gradient stable linear scheme \cite{Eyre} 
\begin{equation}\label{approx3}
u_j^{n+1} - u_j^{n} =  \Delta t D_h (u_j^{n})^3 -3\Delta t D_h u_j^{n} -\epsilon^2\Delta t D_h^2 u_j^{n+1} + 2\Delta t u_j^{n+1}.
\end{equation}

Now to employ the discrete Parareal method for the CH equation \eqref{CH} we denote ${U}_{n}^{k}$ as the approximation at $k$-th iteration containing $u(jh, T_n), j=2,3,...,N_x-1$, where $h$ is the spatial mesh size and $N_x$ is the number of discrete nodes in spatial domain. Now depending on the choice of coarse and fine operator we propose the following five versions of Parareal algorithms for the CH equation \eqref{CH}:
\begin{enumerate}
\item We fix both the fine propagator $\mathcal{F}$ and coarse propagator $\mathcal{G}$ to be the linear scheme in \eqref{approx2} in the Parareal iteration \eqref{parareal}. We call this algorithm PA-I.
\item We fix both the fine propagator $\mathcal{F}$ and coarse propagator $\mathcal{G}$ to be the linear scheme in \eqref{approx3} in the Parareal iteration \eqref{parareal}. We call this algorithm PA-II.
\item We fix the fine propagator $\mathcal{F}$ to be the linear scheme in \eqref{approx3} and coarse propagator $\mathcal{G}$ to be the linear scheme in \eqref{approx2} in the Parareal iteration \eqref{parareal}. We call this algorithm PA-III.
\item We fix the fine propagator $\mathcal{F}$ to be the nonlinear scheme in \eqref{approx1} and coarse propagator $\mathcal{G}$ to be the linear scheme in \eqref{approx2} in the Parareal iteration \eqref{parareal}. We call this algorithm NPA-I.
\item We fix both the fine propagator $\mathcal{F}$ and coarse propagator $\mathcal{G}$ to be the nonlinear scheme in \eqref{approx1} in the Parareal iteration \eqref{parareal}. We call this algorithm NPA-II.
\end{enumerate}
The first three algorithms are linear whereas the last two algorithms are nonlinear as either the fine solver or the coarse solver or both involve nonlinear scheme. Next we discuss the stability and convergence properties of the proposed Parareal algorithms.
\section{Stability and Convergence}\label{Section2}
First we rewrite the fine and coarse propagators in simplified operator form. For the approximation in \eqref{approx2} we have 
\begin{equation}\label{approx2_op_form}
\frac{W^{n+1} -W^{n}}{\Delta t} = D_h \diag(W^n)^2W^{n+1} - \epsilon^2D_h^2 W^{n+1}- D_hW^{n},
\end{equation}
where $W\in\mathbb{R}^{(N_x-2)}$ and the discrete Laplacian $D_h$ with Dirichlet boundary condition is the following
\begin{equation}\label{D_h}
D_h= \frac{1}{h ^2}
\begin{bmatrix}
    -2  & 1  \\
    1  & -2& 1   \\
    &\ddots&\ddots &\ddots \\
     & & 1  & -2& 1   \\
    & & & 1 & -2    
\end{bmatrix} \in\mathbb{R}^{(N_x-2) \times (N_x-2)}.
\end{equation}
Numerical tests suggest that the term $\diag(W^n)^2$ behaves as $I$, away from interface region, which also observed in \cite{David}. Thus for analysing purpose we consider $\diag(W^n)^2 \approx I$.  
Then \eqref{approx2_op_form} can be written as
$W^{n+1}= \left(I-\Delta t D_h +\epsilon^2\Delta t D_h^2\right)^{-1}\left(I - \Delta t D_h\right)W^{n},$
where $I$ is the identity matrix. Then the fine and coarse propagator corresponding to the scheme \eqref{approx2} can be written as 
\begin{subequations}\label{fine_coarse_op1}
\begin{align}
\mathcal{F}(T_{n+1}, T_n,U)&=\left(I-\Delta t D_h +\epsilon^2\Delta t D_h^2\right)^{-1}\left(I - \Delta t D_h\right) U, \; U\in\mathbb{R}^{N_x-2}, \label{fine_approx2}\\
\mathcal{G}(T_{n+1}, T_n,U)&=\left(I-\Delta T D_h +\epsilon^2\Delta T D_h^2\right)^{-1}\left(I - \Delta T D_h\right) U, \; U\in\mathbb{R}^{N_x-2},\label{coarse_approx2}
\end{align}
\end{subequations} 
respectively. Similarly we can write the fine and coarse propagator corresponding to the scheme \eqref{approx3} as 
\begin{subequations}\label{fine_coarse_op2}
\begin{align}
\mathcal{F}(T_{n+1}, T_n,U)&=\left(I-2\Delta t D_h +\epsilon^2\Delta t D_h^2\right)^{-1}\left(I - 2\Delta t D_h\right) U, \; U\in\mathbb{R}^{N_x-2}, \label{fine_approx3}\\
\mathcal{G}(T_{n+1}, T_n,U)&=\left(I-2\Delta T D_h +\epsilon^2\Delta T D_h^2\right)^{-1}\left(I - 2\Delta T D_h\right) U, \; U\in\mathbb{R}^{N_x-2}.\label{coarse_approx3}
\end{align}
\end{subequations} 
The matrix $D_h$ in \eqref{D_h} is symmetric negative definite; 
eigenvalues of $D_h$ are 
$
\lambda_p=\frac{2}{h^2}\left\lbrace \cos\left( \frac{p\pi}{N_x -1}\right)-1 \right\rbrace, p=1,\cdots , N_x-2. 
$
 $\lambda_p$'s are distinct and satisfy $\lambda_p< 0, \forall p$.
Now we define few matrices that we use later in the paper. The matrices are $P_i:=\left(I - i\Delta T D_h+ \epsilon^2\Delta T D_h^2\right)^{-1}\left(I - i\Delta T D_h \right)$ for $i=1, 2, 3$ and    
$ P_{J_i}:=\left[\left(I - \frac{i\Delta T}{J} D_h+ \epsilon^2\frac{\Delta T}{J} D_h^2\right)^{-1}\left(I - \frac{i\Delta T}{J} D_h \right)\right]^J$ for $i=1, 2$.
Before stating the stability and convergence results for the linear Parareal algorithms we first state and prove some auxiliary results.
\begin{lemma}\label{contra_lemma}
Let $J\in\mathbb{N}$ such that $J\geq 2, \Delta T>0, \epsilon>0$ and $y\in(0, \infty)$. Then the functions 
$g_i(y)=\frac{1+i\Delta T y}{1+i\Delta T y+\epsilon^2\Delta T y^2}$ for $i=1, 2, 3$ satisfy $g_i(y)\in(0, 1), \forall y$.

\end{lemma}
\begin{proof}
It is clear that each $g_i$ is continuous in $(0, \infty)$, and $\lim\limits_{y\rightarrow 0^+}g_i(y)=1$ and $\lim\limits_{y\rightarrow \infty}g_i(y)=0$, so we have $g_i(y)\in(0, 1),\forall y$.
%
%
\end{proof}
\begin{lemma}\label{contra_lemma2}
Let $J\in\mathbb{N}$ such that $J\geq 2, \Delta T>0, \epsilon>0$ and $y\in(0, \infty)$. Then the followings hold 
\begin{enumerate}
\item[(i)] 
for $i=1, 2$ the function $\phi_i(y):= \left( g_i(y/J)\right) ^J-g_i(y)$ satisfies $\vert\phi_i(y)\vert<1, \forall y$. \label{st1}
\item[(ii)] 
the function $\phi_3(y):= \left( g_2(y/J)\right) ^J-g_1(y)$ satisfies $\vert\phi_3(y)\vert<1, \forall y$.\label{st2}
\end{enumerate}
\end{lemma}
\begin{proof}
First we prove the statement ($i$). We have $\vert \phi_i (y)\vert <1 \iff \left( g_i(y/J)\right) ^J< 1+g_i(y) \;\&\; \left( g_i(y/J)\right) ^J>-1+g_i(y)$. Using Lemma \ref{contra_lemma} we have $g_i^J(y/J)\in(0, 1)$ and thus we have  $g_i^J(y/J)< 1+g_i(y), \forall y$. The term $-1+g_i(y)=\frac{-\epsilon^2\Delta Ty^2}{1+i\Delta T y+\epsilon^2\Delta Ty^2}<0$, hence $\left( g(y/J)\right) ^J>-1+g_i(y)$ holds. Similarly we can get the result ($ii$). 
\end{proof}

\begin{lemma}[Matrix inverse]\label{mat_inv}
Let $n\in \mathbb{N}, \beta\in\mathbb{R}^+$ then 
\begin{equation}\label{inv}
M(\beta)^{-1}:=
\begin{bmatrix}
    1  & 0  \\
    -\beta & 1& 0   \\
    &\ddots&\ddots &\ddots \\
     & & -\beta  & 1& 0   \\
    & & & -\beta & 1    
\end{bmatrix}^{-1}_{n \times n} 
=
\begin{bmatrix}
    1  & 0  \\
    \beta & 1& 0   \\
    \vdots & &\ddots &\ddots \\
     \beta^{n-2}& \beta^{n-3}& \cdots  & 1 &0  \\
    \beta^{n-1}& \beta^{n-2} & \cdots & \beta & 1    
\end{bmatrix}_{n \times n} 
\end{equation}
\end{lemma}
\begin{proof}
We prove the result \eqref{inv} by induction. Clearly the statement is true for $n=2$. Let us assume that the result \eqref{inv} is true for $n=l$. Then for $n=l+1$ the matrix $M(\beta)$ can be written as the following block form 
\begin{equation}
M(\beta)=
\begin{bmatrix}
A_{l\times l} & B_{l\times 1}\\
C_{1\times l} & D_{1\times 1}
\end{bmatrix}_{(l+1)\times(l+1)},
\end{equation}
where $A=M(\beta)_{l\times l}, B=[\textbf{0}]_{l\times 1}, C=[\textbf{0}, -\beta]_{1\times l}, D=[1]_{1\times 1}$. As we know the inverse of $A$ we have 
\begin{equation}
M(\beta)^{-1}_{(l+1)\times(l+1)}=
\begin{bmatrix}
A^{-1} + A^{-1}BS^{-1}CA^{-1} & -A^{-1}BS^{-1}\\
-S^{-1}CA^{-1} & S^{-1}
\end{bmatrix},
\end{equation}
where $S=D-CA^{-1}B$. Clearly $S^{-1}=[1]_{1\times 1}$, and thus we have $A^{-1}BS^{-1}CA^{-1}=[\textbf{0}]_{l\times l}, -A^{-1}BS^{-1}=[\textbf{0}]_{l\times 1}$, and $-S^{-1}CA^{-1}=[\beta^l, \beta^{l-1},\cdots,\beta]_{1\times l}$. Hence we have the Lemma.
\end{proof}

\begin{lemma}[Matrix power]\label{mat_power}
Let $\beta>0$ and $\mathbb{T}(\beta)$ be a strictly lower triangular Toeplitz matrix of size $N$ whose elements are defined by its first column   
\begin{equation*}
\mathbb{T}_{i,1}=
\begin{cases}
     0  \;\;\text{if}\;\; i=1,\\
     \beta^{i-2}  \;\;\text{if}\;\; 2\leq i \leq N . 
\end{cases}
\end{equation*}
Then the $i$-th element of the first column of the $k$-th power of $\mathbb{T}$ is 
\begin{equation*}
\mathbb{T}_{i,1}^k=
\begin{cases}
     0  \qquad\qquad \quad\text{if}\;\; 1\leq i \leq k,\\
     \binom{i-2}{k-1}\beta^{i-1-k}  \;\;\;\text{if}\;\; k+1\leq i \leq N . 
\end{cases}
\end{equation*}
\end{lemma}
\begin{proof}
See \cite{gander2007analysis}.
\end{proof}
\begin{lemma}\label{mat_norm}
For $0<\beta<1$ the infinity norm of $\mathbb{T}(\beta)^k$ is given by 

$\parallel\mathbb{T}^k\parallel_{\infty} \leq \min \left\{ \left(\frac{1-\beta^{N-1}}{1-\beta}\right)^k,  \binom{N-1}{k} \right\}$
\end{lemma}
\begin{proof}
See \cite{gander2007analysis}.
\end{proof}
\begin{theorem}[Stability of PA-I]\label{thm1}
The algorithm PA-I is stable, i.e., for each $n$ and $k$ the Parareal iteration satisfies $\parallel U_{n+1}^{k+1}\parallel \leq  \parallel u^{0}\parallel + (n+1)\left(\max\limits_{0\leq j\leq n} \parallel {U}_{ j}^{k}\parallel \right)$.
\end{theorem}
\begin{proof}
Using the fine propagator \eqref{fine_approx2} and coarse propagator \eqref{coarse_approx2} in the Parareal scheme \eqref{parareal} we have 
\begin{equation}\label{stab1}
\begin{aligned}[b]
{U}_{n+1}^{k+1}& =P_1{U}_{n}^{k+1} + \left( P_{J_1} - P_1\right){U}_{n}^{k}, \\
\xRightarrow{\text{by taking norm}} \parallel {U}_{n+1}^{k+1}\parallel &  \leq \parallel P_1\parallel \parallel {U}_{n}^{k+1}\parallel + \parallel  P_{J_1} - P_1 \parallel \parallel {U}_{n}^{k}\parallel\\
& \leq \parallel {U}_{n}^{k+1}\parallel + \parallel {U}_{n}^{k}\parallel,
\end{aligned}
\end{equation}
where in the second inequality \eqref{stab1} we use Lemma \ref{contra_lemma} \& \ref{contra_lemma2}.
By the repeated application of the recurrence in \eqref{stab1} for $n$ and taking the sum we have 
\begin{equation*}
\parallel {U}_{n+1}^{k+1}\parallel -\parallel {U}_{0}^{k+1}\parallel   \leq   \sum_{j=0}^n \parallel {U}_{j}^{k}\parallel \leq  (n+1)\left(\max\limits_{0\leq j\leq n} \parallel {U}_{ j}^{k}\parallel \right).
\end{equation*}
Then using $U_0^{k+1}=u^0$ we get the stated result.
\end{proof}
\begin{theorem}[Convergence of PA-I]\label{thm2}
The algorithm PA-I is convergent, i.e., for the error ${E}_{n+1}^{k+1}=U(T_{n+1})-{U}_{n+1}^{k+1}$ the algorithm PA-I satisfies the following error estimate  $\max\limits_{1\leq j\leq N}\parallel {E}_{j}^{k+1}\parallel \leq \alpha^{k+1}\min \left\{ \left(\frac{1-\beta^{N-1}}{1-\beta}\right)^{k+1},  \binom{N-1}{k+1} \right\} \max\limits_{1\leq j\leq N}\parallel {E}_{j}^{0}\parallel$, where $\alpha=\parallel P_{J_1} -P_1 \parallel, \beta=\parallel P_1 \parallel$.
\end{theorem}
\begin{proof}
From the Parareal scheme \eqref{parareal} we have 
\begin{equation}\label{err1}
\begin{aligned}
U(T_{n+1})-{U}_{n+1}^{k+1}& = U(T_{n+1})-\mathcal{G}(T_{n+1},T_n,U_n^{k+1})-\mathcal{F}(T_{n+1},T_n,U_n^{k})+ \mathcal{G}(T_{n+1},T_n,U_n^{k})\\
&= \mathcal{F}(T_{n+1}, T_n, {U}_n) - \mathcal{G}(T_{n+1}, T_n, {U}_n)\\
& -\left( \mathcal{F}(T_{n+1}, T_n, {U}_n^k)- \mathcal{G}(T_{n+1}, T_n, {U}_n^{k}) \right) \\
& +\mathcal{G}(T_{n+1}, T_n, {U}_n)-\mathcal{G}(T_{n+1}, T_n, {U}_n^{k+1}).
\end{aligned}
\end{equation}
Using the fine propagator \eqref{fine_approx2} and coarse propagator \eqref{coarse_approx2} in \eqref{err1} we have the recurrence relation for the error ${E}_{n+1}^{k+1}$  as
\begin{equation}\label{err2}
\begin{aligned}
{E}_{n+1}^{k+1}& =\left( P_{J_1} - P_1\right){E}_{n}^{k} + P_1{E}_{n}^{k+1},\\
\xRightarrow{\text{by taking norm}} \parallel {E}_{n+1}^{k+1} \parallel & \leq \parallel P_{J_1} - P_1\parallel \parallel {E}_{n}^{k}\parallel + \parallel P_1\parallel  \parallel {E}_{n}^{k+1}\parallel.
\end{aligned}
\end{equation}
The recurrence relation in \eqref{err2} can be written in the following matrix form 
\begin{equation}\label{err_matrix}
\begin{aligned}[b]
\begin{bmatrix}
\parallel {E}_{1} \parallel\\
\parallel {E}_{2} \parallel\\
\vdots\\
\parallel {E}_{N-1} \parallel\\
\parallel {E}_{N} \parallel
\end{bmatrix}^{k+1}
& \leq 
\begin{bmatrix}
    1  & 0  \\
    -\beta & 1& 0   \\
    &\ddots&\ddots &\ddots \\
     & & -\beta  & 1& 0   \\
    & & & -\beta & 1    
\end{bmatrix}^{-1}
\begin{bmatrix}
    0  & 0  \\
    \alpha & 0& 0   \\
    &\ddots&\ddots &\ddots \\
     & & \alpha  & 0& 0   \\
    & & & \alpha & 0    
\end{bmatrix}
\begin{bmatrix}
\parallel {E}_{1} \parallel\\
\parallel {E}_{2} \parallel\\
\vdots\\
\parallel {E}_{N-1} \parallel\\
\parallel {E}_{N} \parallel
\end{bmatrix}^{k}\\
&\leq \alpha 
\begin{bmatrix}
    0  & 0  \\
    1 & 0& 0   \\
    \vdots & &\ddots &\ddots \\
     \beta^{N-3}& \beta^{N-4}& \cdots  & 0 &0  \\
    \beta^{N-2}& \beta^{N-3} & \cdots & 1 & 0    
\end{bmatrix}
\begin{bmatrix}
\parallel {E}_{1} \parallel\\
\parallel {E}_{2} \parallel\\
\vdots\\
\parallel {E}_{N-1} \parallel\\
\parallel {E}_{N} \parallel
\end{bmatrix}^{k},
\end{aligned}
\end{equation}
where $\alpha=\parallel P_{J_1} -P_1 \parallel, \beta=\parallel P_1 \parallel$, and on second inequality we use the Lemma \ref{mat_inv}. Now observe that the iteration matrix appearing in \eqref{err_matrix} is Nilpotent, so for $k=N$ we have finite step convergence. Using Lemma \ref{mat_power} and Lemma \ref{mat_norm} in \eqref{err_matrix} we get the stated error contraction relation.
\end{proof}
\begin{theorem}[Stability of PA-II]\label{thm3}
The algorithm PA-II is stable, i.e., for each $n$ and $k$ the Parareal iteration satisfies $\parallel U_{n+1}^{k+1}\parallel \leq  \parallel u^{0}\parallel+ (n+1)\left(\max\limits_{0\leq j\leq n} \parallel {U}_{ j}^{k}\parallel \right)$.
\end{theorem}
\begin{proof}
Emulating the proof of Theorem \ref{thm1} we have the stated result.
\end{proof}
\begin{theorem}[Convergence of PA-II]\label{thm4}
The algorithm PA-II is convergent, i.e., for the error ${E}_{n+1}^{k+1}=U(T_{n+1})-{U}_{n+1}^{k+1}$ the algorithm PA-II satisfies the following error estimate  $\max\limits_{1\leq j\leq N}\parallel {E}_{j}^{k+1}\parallel \leq \alpha^{k+1}\min \left\{ \left(\frac{1-\beta^{N-1}}{1-\beta}\right)^{k+1},  \binom{N-1}{k+1} \right\} \max\limits_{1\leq j\leq N}\parallel {E}_{j}^{0}\parallel$, where $\alpha=\parallel P_{J_2}-P_2 \parallel,  \beta= \parallel P_2 \parallel$.
\end{theorem}
\begin{proof}
The proof follows from the Theorem \ref{thm2}.
\end{proof}

\begin{theorem}[Stability of PA-III]\label{thm5}
The algorithm PA-III is stable, i.e., for each $n$ and $k$ the Parareal iteration satisfies $\parallel U_{n+1}^{k+1}\parallel \leq  \parallel u^{0}\parallel+ (n+1)\left(\max\limits_{0\leq j\leq n} \parallel {U}_{ j}^{k}\parallel \right)$.
\end{theorem}
\begin{proof}
Emulating the proof of Theorem \ref{thm1} we have the result.
\end{proof}
\begin{theorem}[Convergence of PA-III]\label{thm6}
The algorithm PA-III is convergent, i.e., for the error ${E}_{n+1}^{k+1}=U(T_{n+1})-{U}_{n+1}^{k+1}$ the algorithm PA-III satisfies the following error estimation  $\max\limits_{1\leq j\leq N}\parallel {E}_{j}^{k+1}\parallel \leq \alpha^{k+1}\min \left\{ \left(\frac{1-\beta^{N-1}}{1-\beta}\right)^{k+1},  \binom{N-1}{k+1} \right\} \max\limits_{1\leq j\leq N}\parallel {E}_{j}^{0}\parallel$, where $\alpha=\parallel P_{J_2}-P_1 \parallel,  \beta= \parallel P_1 \parallel$.
\end{theorem}
\begin{proof}
The proof follows from the Theorem \ref{thm2}.
\end{proof}

Next we prove a few relevant results before discussing the stability and convergence of nonlinear Parareal method.
\begin{lemma}[Growth of Coarse Operator in NPA-I]\label{growth_G}
The coarse operator in \eqref{coarse_approx2} satisfies the growth condition 
$\parallel \mathcal{G}(T_{n+1}, T_n, U)\parallel  \leq \parallel U \parallel, \forall U\in\mathbb{R}^{N_x-2}$. 
\end{lemma}
\begin{proof}
We have $\parallel \mathcal{G}(T_{n+1}, T_n, U)\parallel \leq  \parallel P_1\parallel \parallel U \parallel$. Now $\parallel P_1\parallel<1$ follows from Lemma \ref{contra_lemma}, hence the result.
\end{proof}

\begin{lemma}[Lipschitz Property of $\mathcal{G}$]\label{Lip_G}
The coarse operator in \eqref{coarse_approx2} satisfies the Lipschitz condition 
$
\parallel \mathcal{G}(T_{n+1}, T_n, {U_1})-\mathcal{G}(T_{n+1}, T_n, {U_2})\parallel  \leq \parallel P_1\parallel \parallel {U_1}-{U_2}\parallel, \forall U_1, U_2\in\mathbb{R}^{N_x-2}.
$ 
\end{lemma}
\begin{proof} The result is straight forward.
\end{proof}

\begin{lemma}[Local Truncation Error (LTE) Differences in NPA-I]\label{LTE}
Let $\mathcal{F}(T_{n+1}, T_n, {U})$ be the fine operator generated by the nonlinear scheme in \eqref{approx1} and $\mathcal{G}(T_{n+1}, T_n, {U})$ be the coarse operator in \eqref{coarse_approx2}. Then the following LTE differences hold  
$$
\mathcal{F}(T_{n+1}, T_n, {U})-\mathcal{G}(T_{n+1}, T_n, {U}) =c_2(U)\Delta T ^2 +c_3(U)\Delta T ^3 + \cdots ,$$ 
where $c_j(U)$ are continuously differentiable function for $j=2, 3, ...$
\end{lemma}
\begin{proof}
Let $\mathcal{S}(T_{n+1}, T_n, {U})$ be the exact solution of \eqref{CH}, then 
\begin{equation*}
\begin{aligned}
\mathcal{F}(T_{n+1}, T_n, {U})-\mathcal{G}(T_{n+1}, T_n, {U}) & =\mathcal{F}(T_{n+1}, T_n, {U})-\mathcal{S}(T_{n+1}, T_n, {U})\\ & + \mathcal{S}(T_{n+1}, T_n, {U})-\mathcal{G}(T_{n+1}, T_n, {U})\\
& = \tilde{c_2}(U)\Delta T^2 +\tilde{c_3}(U)\Delta T ^3 + \cdots \\ & +\hat{c_2}(U)\Delta T^2 +\hat{c_3}(U)\Delta T ^3 + \cdots\\
& = c_2(U)\Delta T^2 +c_3(U)\Delta T ^3 + \cdots .
\end{aligned}
\end{equation*}
Hence the Lemma.
\end{proof}

\begin{theorem}[Stability of NPA-I]\label{thm_stab_npa1}
The algorithm NPA-I is stable, i.e., for each $n$ and $k$, $ 
\parallel {U}_{n+1}^{k+1}\parallel \leq \parallel u^0\parallel + C (n+1)\Delta T^2\left(\max\limits_{0\leq j\leq n} \parallel {U}_{ j}^{k}\parallel \right)$, for a constant $C$.
\end{theorem}
\begin{proof}
Taking the norm in the correction scheme \eqref{parareal} we have
\begin{equation}\label{rec_relation_1}
\begin{aligned}
\parallel {U}_{n+1}^{k+1}\parallel & \leq \parallel \mathcal{G}(T_{n+1}, T_n, {U}_n^{k+1})\parallel +\parallel \mathcal{F}(T_{n+1}, T_n, {U}_n^k)-\mathcal{G}(T_{n+1}, T_n, {U}_n^k)\parallel \\
 & \leq \parallel {U}_{n}^{k+1}\parallel + C \Delta T^2 \parallel {U}_{n}^{k}\parallel,
\end{aligned}
\end{equation}
where in the 2nd inequality we use Lemma \eqref{growth_G} and \eqref{LTE}. Takin the sum over $n$ on the recurrence relation \eqref{rec_relation_1} we get 
\begin{equation*}
\parallel {U}_{n+1}^{k+1}\parallel -\parallel {U}_{0}^{k+1}\parallel   \leq  C \Delta T^2 \sum_{j=0}^n \parallel {U}_{j}^{k}\parallel \leq C (n+1)\Delta T^2\left(\max\limits_{0\leq j\leq n} \parallel {U}_{ j}^{k}\parallel \right).
\end{equation*}
Then using $U_0^{k+1}=u^0$ we get the stated result.
\end{proof}
\begin{theorem}[Convergence of NPA-I]\label{thm_npa1}
If the propagator $\mathcal{F}$ in \eqref{approx1} and $\mathcal{G}$ in \eqref{approx2} satisfy LTE differences given in Lemma \ref{LTE} and $\mathcal{G}$ satisfy Lipschitz condition given in Lemma \ref{Lip_G}, then the algorithm NPA-I satisfies the following error estimate
$$\max\limits_{1\leq j\leq N}\parallel {E}_{j}^{k+1}\parallel \leq \left( C_1 \Delta T^2\right)^{k+1}\min \left\{ \left(\frac{1-\beta^{N-1}}{1-\beta}\right)^{k+1},  \binom{N-1}{k+1} \right\} \max\limits_{1\leq j\leq N}\parallel {E}_{j}^{0}\parallel,$$
where $\beta=\parallel P_1\parallel$, and $C_1$ is a constant related to LTE.
\end{theorem}
\begin{proof}
From the Parareal scheme \eqref{parareal} we have 
\begin{equation}\label{err_npa1}
\begin{aligned}
U(T_{n+1})-{U}_{n+1}^{k+1}& = U(T_{n+1})-\mathcal{G}(T_{n+1},T_n,U_n^{k+1})-\mathcal{F}(T_{n+1},T_n,U_n^{k})+ \mathcal{G}(T_{n+1},T_n,U_n^{k})\\
&= \mathcal{F}(T_{n+1}, T_n, {U}_n) - \mathcal{G}(T_{n+1}, T_n, {U}_n)\\
& -\left( \mathcal{F}(T_{n+1}, T_n, {U}_n^k)- \mathcal{G}(T_{n+1}, T_n, {U}_n^{k}) \right) \\
& +\mathcal{G}(T_{n+1}, T_n, {U}_n)-\mathcal{G}(T_{n+1}, T_n, {U}_n^{k+1})\\
& =  (c_2(U_n)\Delta T^2 +c_3(U_n)\Delta T ^3 + \cdots)- (c_2(U_n^k)\Delta T^2 +c_3(U_n^k)\Delta T ^3 + \cdots) \\
& +\mathcal{G}(T_{n+1}, T_n, {U}_n)-\mathcal{G}(T_{n+1}, T_n, {U}_n^{k+1}),
\end{aligned}
\end{equation}
where in the third equality we use the Lemma \ref{LTE}. As $c_j, j\geq 2$ are continuously differentiable function we have
\begin{equation}\label{lte_thm}
\begin{aligned}[b]
& \parallel (c_2(U_n)\Delta T^2 +c_3(U_n)\Delta T ^3 + \cdots)- (c_2(U_n^k)\Delta T^2 +c_3(U_n^k)\Delta T ^3 + \cdots) \parallel \\
& \leq \Delta T^2 \parallel c_2(U_n)-c_2(U_n^k) \parallel + \Delta T ^3 \parallel c_3(U_n) - c_3(U_n^k)\parallel + \cdots \\
& \leq C_2\Delta T^2 \parallel U_n -U_n^k \parallel + C_3\Delta T ^3 \parallel U_n -  U_n^k\parallel + \cdots = C_1 \Delta T^2 \parallel U_n -U_n^k \parallel.
\end{aligned}
\end{equation}
Taking norm in \eqref{err_npa1} and using \eqref{lte_thm} and the Lipschitz condition given in Lemma \ref{Lip_G} we have the following recurrence relation for the error $E_{n+1}^{k+1}=U(T_{n+1})-{U}_{n+1}^{k+1}$ as 
\begin{equation}\label{rec2}
\parallel E_{n+1}^{k+1} \parallel \leq C_1 \Delta T^2 \parallel E_{n}^{k} \parallel + \parallel P_1\parallel  \parallel E_{n}^{k+1} \parallel.
\end{equation}
The recurrence relation in \eqref{rec2} can be written in the following matrix form 
\begin{equation}\label{err_iter}
e^{k+1}= C_1 \Delta T^2 \mathbb{T}(\beta)e^{k},
\end{equation}
where $e^{k}=
\begin{bmatrix}
\parallel {E}_{1}^{k} \parallel, \parallel {E}_{2}^{k} \parallel, \cdots \parallel {E}_{N}^{k} \parallel
\end{bmatrix}^t$, and $\beta=\parallel P_1\parallel$. Clearly the iteration matrix in \eqref{err_iter} is Nilpotent, hence we have finite step convergence. Now to get the stated result  we use the Lemma \ref{mat_power} and infinity norm in \eqref{err_iter}.
\end{proof}
Next we discuss the stability and convergence behaviour of the Parareal algorithm NPA-II. In this case both fine and coarse propagators are nonlinear. To get the coarse operator in its explicit form we use Newton method to the nonlinear system. So the solution of the nonlinear coarse operator in \eqref{approx1} is the zeros of the following nonlinear equations
\begin{equation}\label{nonl_eq}
H(Y)=Y+\epsilon^2 \Delta t D_h^2 Y - \Delta t D_hY^3 - \hat{Y},
\end{equation}
where $Y=U^{n+1}\in\mathbb{R}^{N_x-2}$ and $\hat{Y}=(I-\Delta t D_h)U^n\in\mathbb{R}^{N_x-2}$. After applying the Newton method on \eqref{nonl_eq} with iteration index $m$ and then simplifying we have 
\begin{equation}\label{Newton_iter}
Y_{m+1}=(I + \epsilon^2\Delta t D_h^2  - 3\Delta t D_h\diag(Y_m^2))^{-1} ( \hat{Y}-2\Delta t D_h \diag(Y_m^2)Y_m).
\end{equation}
Numerical experiments suggest that upon convergence of the Newton Method the term $\diag(Y_m^2) \approx I$ away from interface region of width $\epsilon$, similar behaviour also observed in \cite{brenner2018robust}. If one uses initial solution as an initial guess for the Newton method then the coarse operator \eqref{approx1} takes the following form 
\begin{equation}\label{co_op_npa2}
\mathcal{G}(T_{n+1}, T_n, U)= \left[(I + \Delta T D_h^2  - 3\Delta T D_h)^{-1} ( I-3\Delta T D_h)\right]^m U,
\end{equation}
for some Newton iteration $m$. Next we prove some auxiliary results.
\begin{lemma}[Growth of Coarse Operator in NPA-II]\label{growth_G_npa2}
The coarse operator in \eqref{co_op_npa2} satisfies the growth condition 
$\parallel \mathcal{G}(T_{n+1}, T_n, U)\parallel  \leq \parallel U \parallel, \forall U\in\mathbb{R}^{N_x-2} $.
\end{lemma}
\begin{proof}
We have $\parallel \mathcal{G}(T_{n+1}, T_n, U)\parallel \leq  \parallel P_3\parallel^m \parallel U \parallel$. Now $\parallel P_3\parallel<1$ follows from Lemma \ref{contra_lemma}, hence the result.
\end{proof}
\begin{lemma}[Lipschitz of $\mathcal{G}$ in NPA-II]\label{Lip_G_npa2}
The coarse operator in \eqref{co_op_npa2} satisfies the Lipschitz condition 
$
\parallel \mathcal{G}(T_{n+1}, T_n, {U_1})-\mathcal{G}(T_{n+1}, T_n, {U_2})\parallel  \leq \parallel P_3\parallel\parallel {U_1}-{U_2}\parallel, \forall U_1, U_2\in\mathbb{R}^{N_x-2}.
$ 
\end{lemma}
\begin{proof} The result is straight forward.
\end{proof}
\begin{lemma}[LTE Differences in NPA-II]\label{LTE2}
Let $\mathcal{F}(T_{n+1}, T_n, {U})$ and $\mathcal{G}(T_{n+1}, T_n, {U})$ be the fine and coarse operator generated by the nonlinear scheme in \eqref{approx1}, then the following LTE differences hold  
$$
\mathcal{F}(T_{n+1}, T_n, {U})-\mathcal{G}(T_{n+1}, T_n, {U}) =c_2(U)\Delta T ^2 +c_3(U)\Delta T ^3 + \cdots ,$$ 
where $c_j(U)$ are continuously differentiable function for $j=2, 3, ...$
\end{lemma}
\begin{proof} The result follows from the Lemma \ref{LTE}.
\end{proof}

\begin{theorem}[Stability of NPA-II]
The algorithm NPA-II is stable, i.e., for each $n$ and $k$,  $ 
\parallel {U}_{n+1}^{k+1}\parallel \leq \parallel u^0\parallel + C (n+1)\Delta T^2\left(\max\limits_{0\leq j\leq n} \parallel {U}_{ j}^{k}\parallel \right)$, for some constant $C$. 
\end{theorem}
\begin{proof}
The proof can be obtained by following the proof of Theorem \ref{thm_stab_npa1}.
\end{proof}
\begin{theorem}[Convergence of NPA-II]\label{thm_npa2}
If the propagator $\mathcal{F}$ and $\mathcal{G}$ in \eqref{approx1} satisfy LTE differences given in Lemma \ref{LTE2} and $\mathcal{G}$ satisfies the Lipschitz condition given in Lemma \ref{Lip_G_npa2}, then the algorithm NPA-II satisfies the following error estimate
$$\max\limits_{1\leq j\leq N}\parallel {E}_{j}^{k+1}\parallel \leq \left( C_2 \Delta T^2\right)^{k+1}\min \left\{ \left(\frac{1-\beta^{N-1}}{1-\beta}\right)^{k+1},  \binom{N-1}{k+1} \right\} \max\limits_{1\leq j\leq N}\parallel {E}_{j}^{0}\parallel,$$
where $\beta=\parallel P_3\parallel$, and $C_2$ is a constant related to LTE.
\end{theorem}
\begin{proof}
The proof is similar to the proof of Theorem \ref{thm_npa1}.
\end{proof}
\remark
\begin{enumerate}
\item
One can obtain the convergence estimate of NPA-II at the semi-discrete level by estimating coarse operator at the semi-discrete level.
\item Explicit expression of the linear and nonlinear Paraeal algorithms in 2D or 3D can be achieved by extending the 1D case naturally. 
\item
Convergence proof of Parareal method in higher dimension follows from the 1D case by deriving the discrete Laplacian $D_h$ for regular or irregular computational domain.
\item 
The term $\binom{N-1}{k+1}=\frac{1}{(k+1)!}\prod_{j=0}^{k-1}(N-j)$ which appears in all of the convergence results says that methods converges at most $N+1$ iteration. So we always have finite step convergence to the fine solution.
\end{enumerate}

\section{Numerical Illustration}\label{Section3}
In this section we present the numerical experiments for the linear and non-linear Parareal algorithms, which are analyzed in this article. The Parareal iterations start with an initial guess given by coarse operator and stop as the error measured in $\parallel U-U^k\parallel_{L^{\infty}(0,T;L^2(\Omega))}$ reaches a tolerance of $10^{-6}$, where $U$ is the discrete fine solution and $U^k$ is the discrete Parareal solution at $k$-th iteration. We consider the spatial domain $\Omega=(0,1)$ in 1D and $\Omega=(0,1)^2$ in 2D.
\input{numerics_paI}
\input{numerics_paII}
\input{numerics_paIII}
\input{numerics_npaI}
\input{numerics_npaII}

\subsection{Numerical Experiments of Neumann-Neumann method as fine solver in PA-I}
In all of the above experiments we use the scheme \eqref{approx1}, \eqref{approx2} or \eqref{approx3} as fine solver. In practise one try to solve the CH equation in much larger domain with very fine mesh, that results in a very large scale algebraic system (as the spatial dimension increases). In this context one introduce parallelism in space by using Domain Decomposition (DD) based techniques, here we use a non-overlapping DD method, namely Neumann-Neumann (NN) method. The NN method for the CH equation in space is considered in \cite{convergence2021garai} for two subdomain decomposition and in \cite{garai2021convergence} for multiple subdomain decomposition, where they use \eqref{approx1} and \eqref{approx2} to build linear and nonlinear NN solver. Here we use linear NN method as fine solver in the PA-I algorithm. In every subinterval $[T_{n-1}, T_{n}]$ we compute the solution as the following:

Let $\Omega \subset\mathbb{R}$ is decomposed into non-overlapping subdomains $\{\Omega_i, 1\leq i \leq N_0\}$.
So to solve \eqref{approx2} at each time level the NN method starts with initial guesses $g_i^{[0]}, h_i^{[0]}$ along the interfaces $\Gamma_i=\partial\Omega_i \cap \partial\Omega_{i+1}$ for $i = 1,\ldots,N_0-1$, and then it's a two step execution: at each iteration $\nu$, one first solves Dirichlet sub-problems on each $\Omega_i$ in parallel, and then compute the jump in Neumann traces on the interfaces  and one solves the Neumann subproblems on each $\Omega_i$ in parallel,
\begin{equation}\label{MNCH}
\left\{
\begin{aligned}
\begin{bmatrix} 
I & -\delta_t\Delta \\
\epsilon^2\Delta-c^2 & I 
\end{bmatrix}
\quad
\begin{bmatrix} 
u_i^{[\nu]} \\
v_i^{[\nu]} 
\end{bmatrix}
& =  \begin{bmatrix} 
f_u \\
f_v 
\end{bmatrix},
\quad \mbox{in}\,\ \Omega_i,\\
\begin{bmatrix} 
u_i^{[\nu]} \\
v_i^{[\nu]} 
\end{bmatrix}
& =  0,\quad \mbox{on}\,\ \partial\Omega_i\cap\partial\Omega,\\
\begin{bmatrix} 
u_i^{[\nu]} \\
v_i^{[\nu]} 
\end{bmatrix}
 & = 
\begin{bmatrix} 
g_{i-1}^{[\nu-1]} \\
h_{i-1}^{[\nu-1]} 
\end{bmatrix}
\quad \mbox{on}\,\ \Gamma_{i-1},\\
\begin{bmatrix} 
u_i^{[\nu]} \\
v_i^{[\nu]} 
\end{bmatrix}
 & = 
\begin{bmatrix} 
g_{i}^{[\nu-1]} \\
h_{i}^{[\nu-1]} 
\end{bmatrix}
\quad \mbox{on}\,\ \Gamma_{i},\\
\end{aligned}\right.
\end{equation}
\begin{equation}\label{NN_Nstep}
\left\{
\begin{aligned}
\begin{bmatrix} 
I & -\delta_t\Delta \\
\epsilon^2\Delta-c^2 & I 
\end{bmatrix}
\begin{bmatrix} 
\phi_i^{[\nu]} \\
\psi_i^{[\nu]} 
\end{bmatrix} 
& =  0,
\quad \mbox{in}\,\ \Omega_i,\\
\begin{bmatrix} 
\phi_i^{[\nu]} \\
\psi_i^{[\nu]} 
\end{bmatrix}
 & =  0,\quad \mbox{on}\,\ \partial\Omega_i\cap\partial\Omega,\\
 \frac{\partial}{\partial x}\begin{bmatrix} 
\phi_i^{[\nu]} \\
\psi_i^{[\nu]} 
\end{bmatrix} & = 
\frac{\partial}{\partial x}\begin{bmatrix} 
u_{i-1}^{[\nu]} - u_{i}^{[\nu]} \\
v_{i-1}^{[\nu]} - v_{i}^{[\nu]} 
\end{bmatrix},
\quad \mbox{on}\,\ \Gamma_{i-1}\\
\frac{\partial}{\partial x}\begin{bmatrix} 
\phi_i^{[\nu]} \\
\psi_i^{[\nu]} 
\end{bmatrix} & = 
\frac{\partial}{\partial x}\begin{bmatrix} 
u_i^{[\nu]} - u_{i+1}^{[\nu]} \\
v_i^{[\nu]} - v_{i+1}^{[\nu]} 
\end{bmatrix},
\quad \mbox{on}\,\ \Gamma_i,
\end{aligned}\right. 
\end{equation}

\noindent
except for the first and last subdomains, where at the physical boundaries the Dirichlet condition in the Dirichlet step and Neumann condition in the Neumann step are replaced by homogeneous Dirichlet condition. 
Then the interface traces are updated by
\[
\begin{bmatrix} 
g_i^{[\nu]} \\
h_i^{[\nu]} 
\end{bmatrix}
 = \begin{bmatrix} 
g_i^{[\nu-1]} \\
h_i^{[\nu-1]} 
\end{bmatrix} -
  \theta \begin{bmatrix} 
\phi_i^{[\nu]} - \phi_{i+1}^{[\nu]} \\
\psi_i^{[\nu]} - \psi_{i+1}^{[\nu]} 
\end{bmatrix}_{\mkern 1mu \vrule height 2ex\mkern2mu {\Gamma_i}}, 
\]
where $\theta\in(0, 1)$ is a relaxation parameter. In NN method \eqref{MNCH}-\eqref{NN_Nstep}, $\delta_t=\Delta t$ is the fine time step, $f_u= u^n, f_v=-u^n, c=(u^n)^2$, where $u^n$ is solution of the CH equation at $n$-th time step. In a similar fashion one can formulate NN method for the scheme given in \eqref{approx3} and use in PA-III as a fine solver. There is also a nonlinear  version of \eqref{MNCH} in \cite{garai2021convergence}, which can be used as a fine solver in the nonlinear Parareal case. To see the numerical experiments in 1D, we take $N_0=8 \;\text{(equal subdomain)}, \theta=1/4$. Note that the parareal solution  converges towards fine solution given by the NN method. For convergence of NN method at each time level we set the tolerance as $\parallel g_i^{[\nu+1]} - g_i^{[\nu]} \parallel_{L^2} \leq 10^{-10}$ and  $\parallel h_i^{[\nu+1]} - h_i^{[\nu]} \parallel_{L^2} \leq 10^{-10}$. The convergence of NN method described in \cite{garai2021convergence}; here we study the convergence of Parareal method PA-I to the NN solution given by \eqref{MNCH}.
We plot the error curves on the left in Figure \ref{diff_T_N_DT} for short as well as long time window with $\epsilon=0.0725, J=200$ and $h=1/128$. The left plot in Figure \ref{diff_T_N_DT} is almost identical to the left plot given in Figure \ref{diff_T_N_eps}. So we have similar convergence behaviour for NN method as fine solver with an advantage of more parallelism in the system. To see the dependency on the parameter $\Delta T$, we plot the error curves on the right in Figure \ref{diff_T_N_DT} for different $\Delta T$ by taking $T=1, J=150, \epsilon=0.0725$. We can observe that convergence is robust. 
\begin{figure}
    \centering
    \subfloat{{\includegraphics[height=3.5cm,width=6cm]{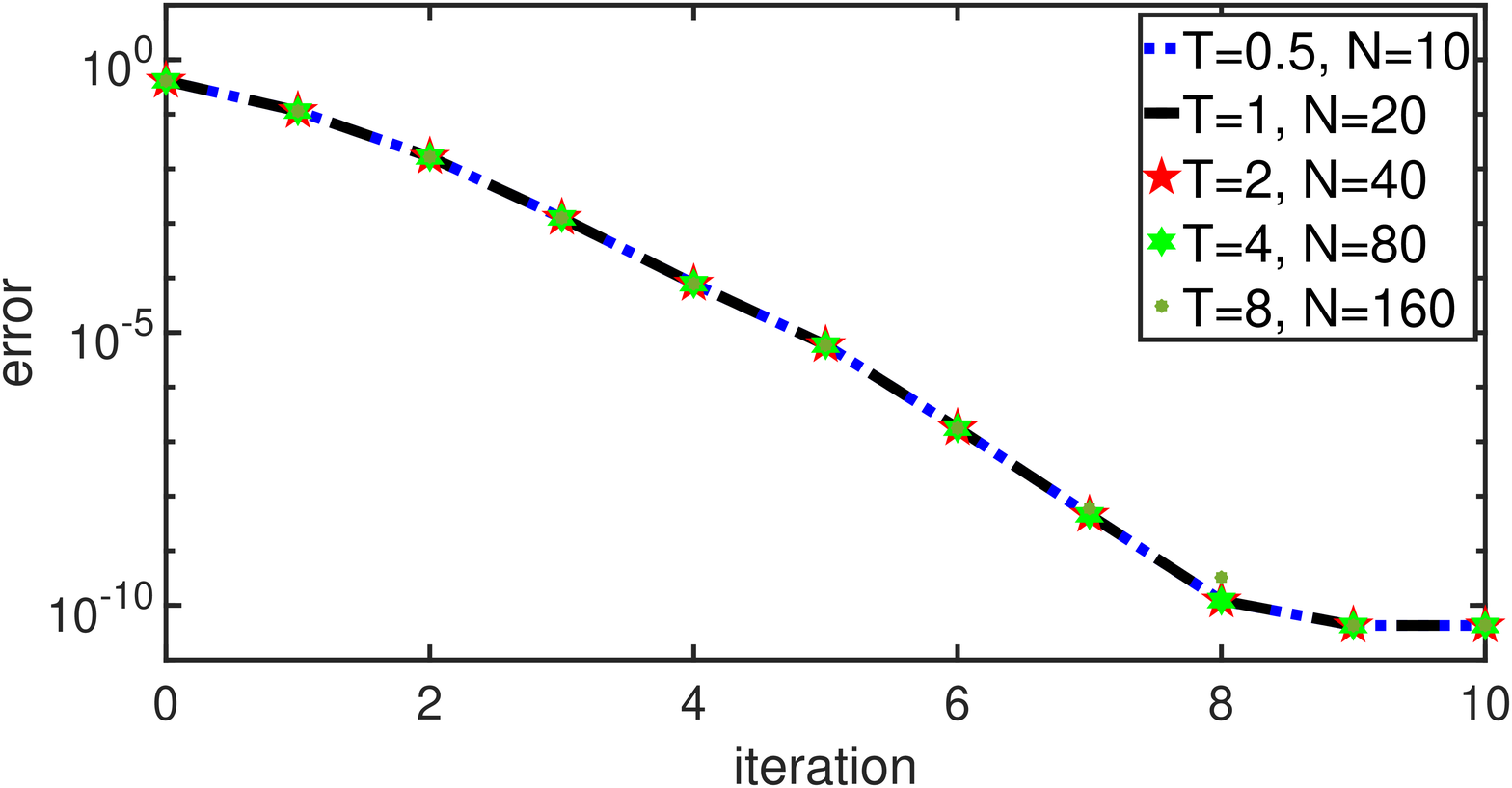} }}
     \subfloat{{\includegraphics[height=3.5cm,width=6cm]{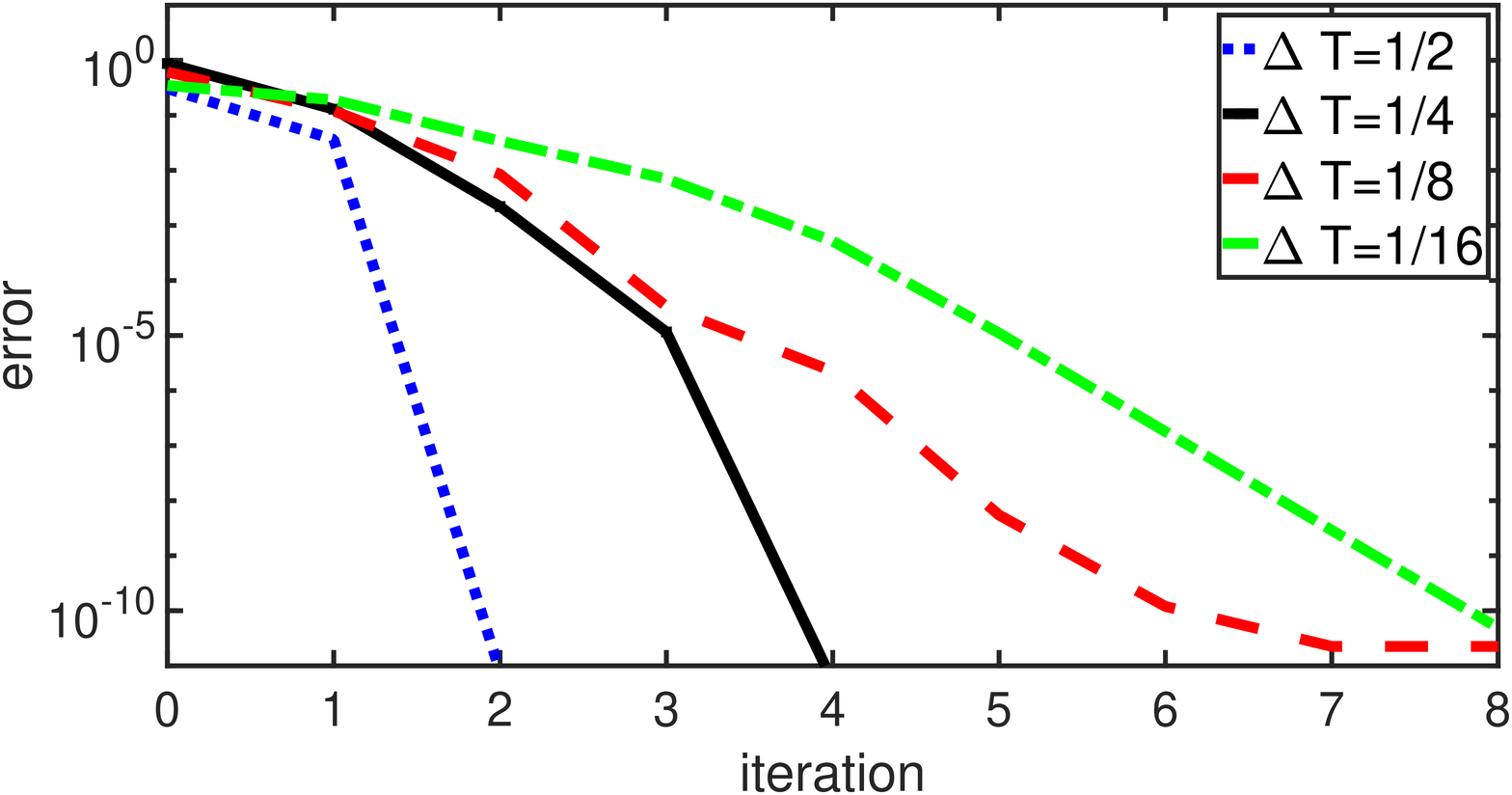} }}
    \caption{On the left: different $T, N$, and on the right: different $\Delta T$.}
    \label{diff_T_N_DT}
\end{figure}

\section{Conclusions}

We propose and studied the linear and nonlinear Parareal algorithms for the CH equation. We showed convergence of all the proposed Parareal algorithms. Numerical experiments show that proposed methods are very robust and one obtains a reasonable speed up by introducing more processor.

\section*{Acknowledgement} The authors would like to thank the CSIR (File No:09/1059(0019)/2018-EMR-I) and DST-SERB (File No: SRG/2019/002164) for the research grant and IIT Bhubaneswar for providing excellent research environment.

\bibliographystyle{siam}
\bibliography{pararealbib}

\end{document}

%% file: numerics_paI.tex
\subsection{Numerical Experiments of PA-I}
We first discuss the numerical experiments in 1D. We run the PA-I algorithm with fixed parameters $T=1, h=1/64, N=20, J=200$ and two different $\epsilon=0.0725, 0.725$. The comparison of theoretical error estimate from Theorem \ref{thm2} and numerical error reduction can be seen in Figure \ref{diff_errbound}. We observe that for larger $\epsilon$ the theoretical bound given in Theorem \ref{thm2} is much sharper than the bound corresponding to smaller $\epsilon$. The reason being is that even though $\alpha<1$ \& $\beta<1$ in Theorem \ref{thm2} for every choice of $\epsilon$, the values of $\alpha, \beta$ increases to one as $\epsilon$ decreases.
\begin{figure}[h]
    \centering
    \subfloat{{\includegraphics[height=3.5cm,width=6cm]{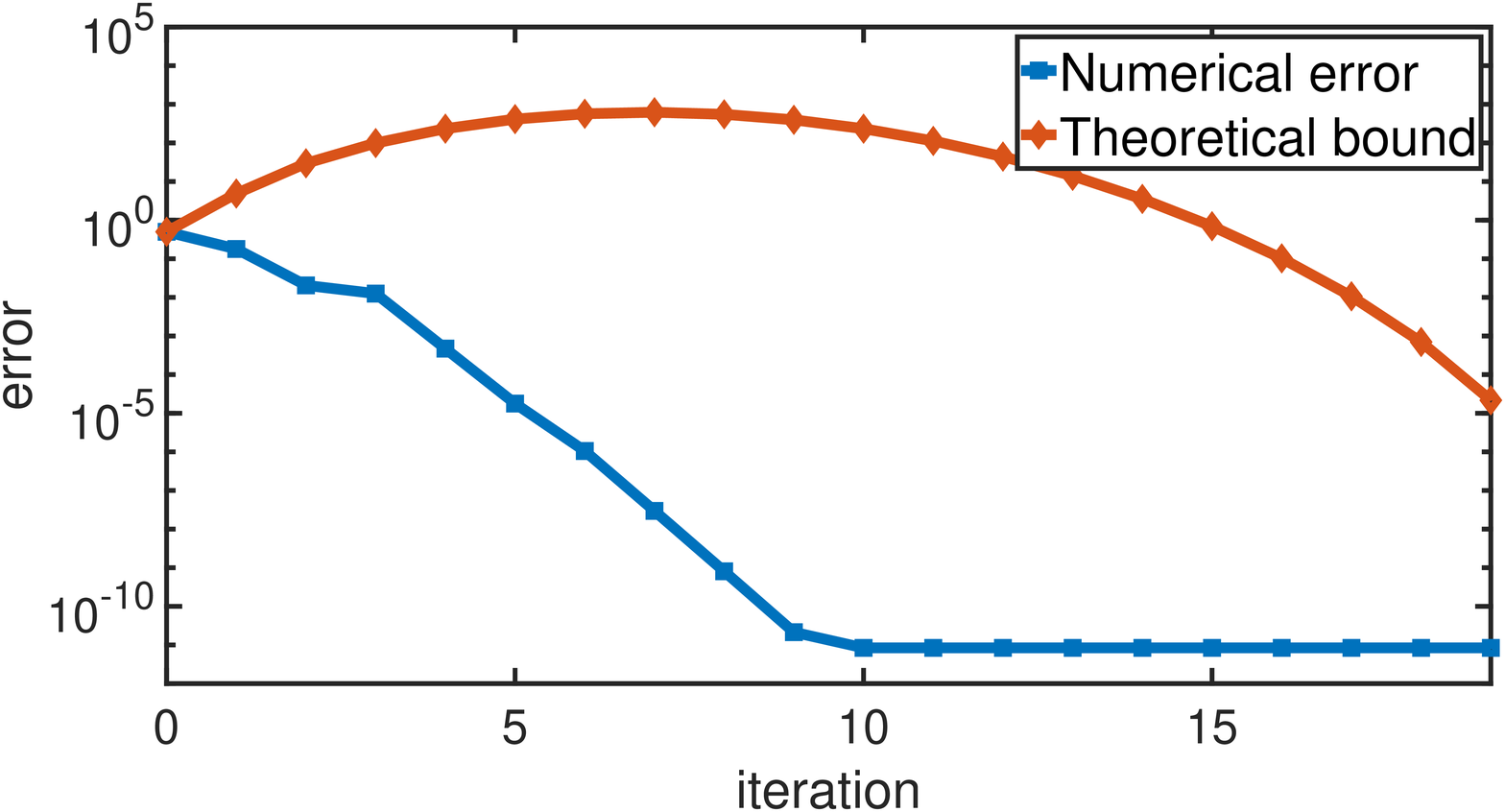} }}
     \subfloat{{\includegraphics[height=3.5cm,width=6cm]{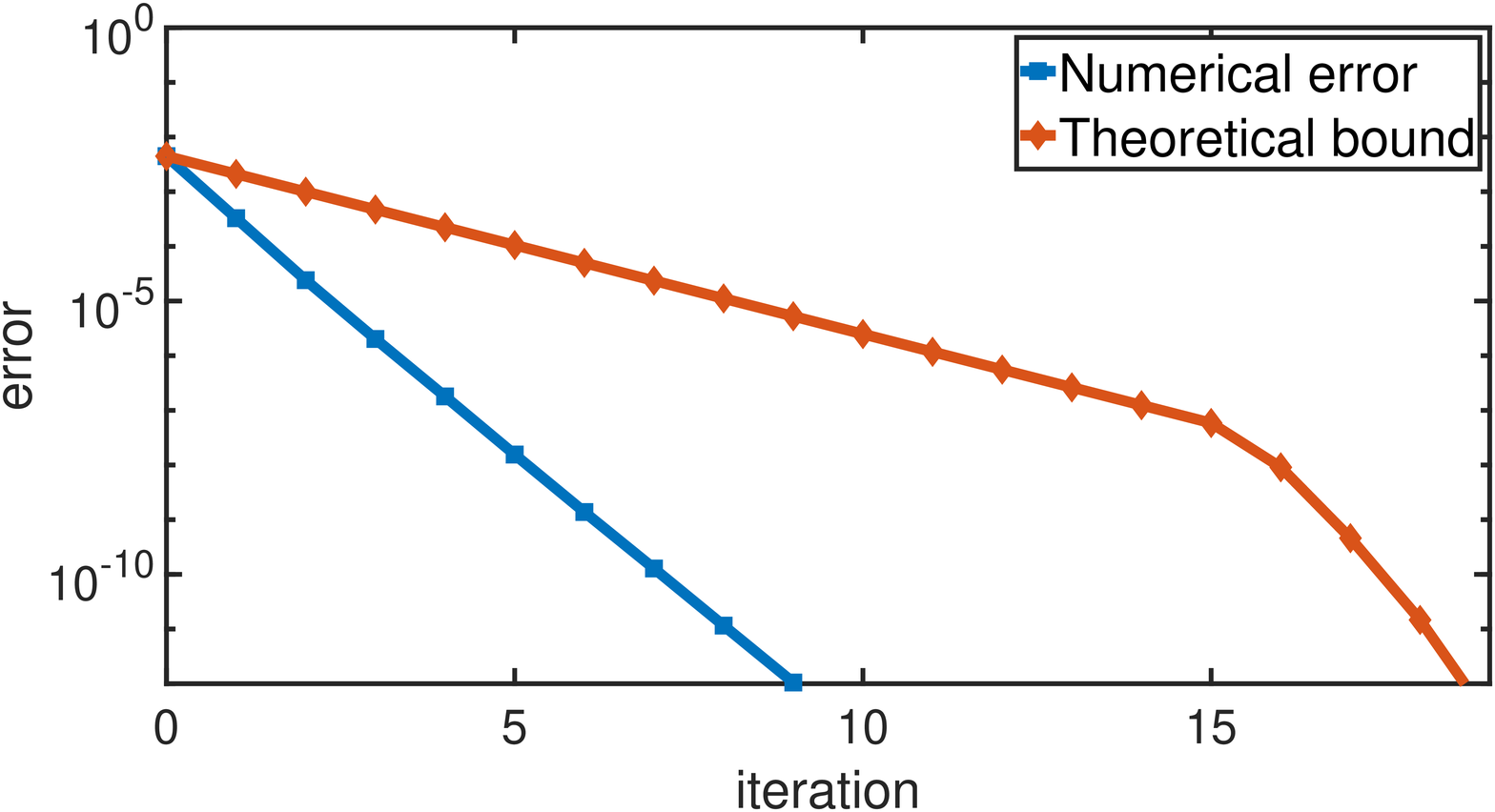} }}
    \caption{PA-I: Comparison of theoretical and numerical error. On the left $\epsilon=0.0725$, and on the right $\epsilon=0.725$.}
    \label{diff_errbound}
\end{figure}
Now we study the convergence behaviour of PA-I on the choice of $\Delta T$. In Figure \ref{diff_dtdx} we plot the error curves for different $\Delta T$ with fixed parameters $\epsilon=0.0725, h=1/64, J=200$ on the left panel and we can see that the method works well for large $\Delta T$. On the right in Figure \ref{diff_dtdx} we plot the error curves for different mesh sizes with $T=1, \epsilon=0.0725, h=1/64, \Delta t=1/200$. We observe that convergence is independent of mesh parameters.
\begin{figure}[h]
    \centering
    \subfloat{{\includegraphics[height=3.5cm,width=6cm]{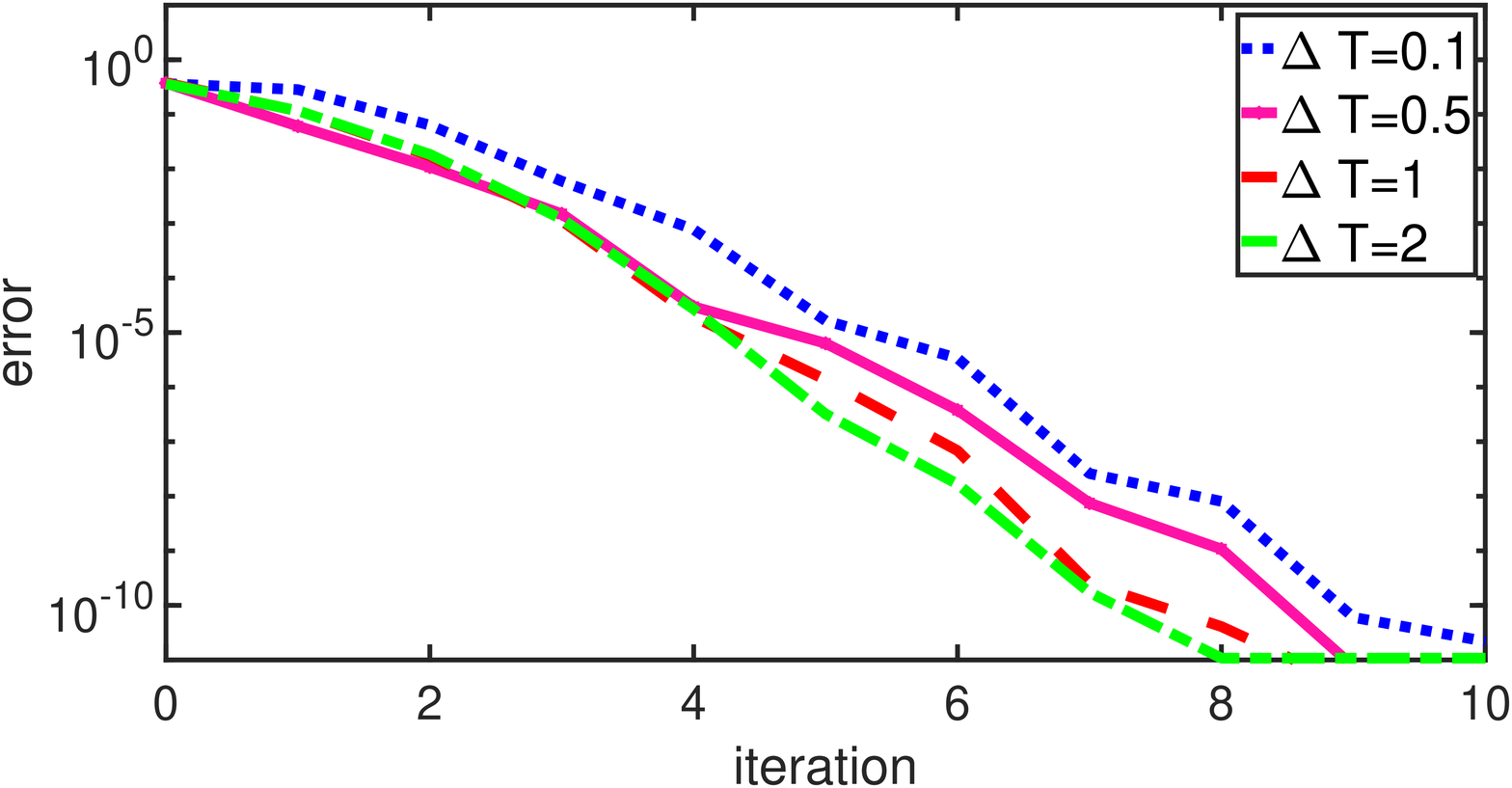} }}
     \subfloat{{\includegraphics[height=3.5cm,width=6cm]{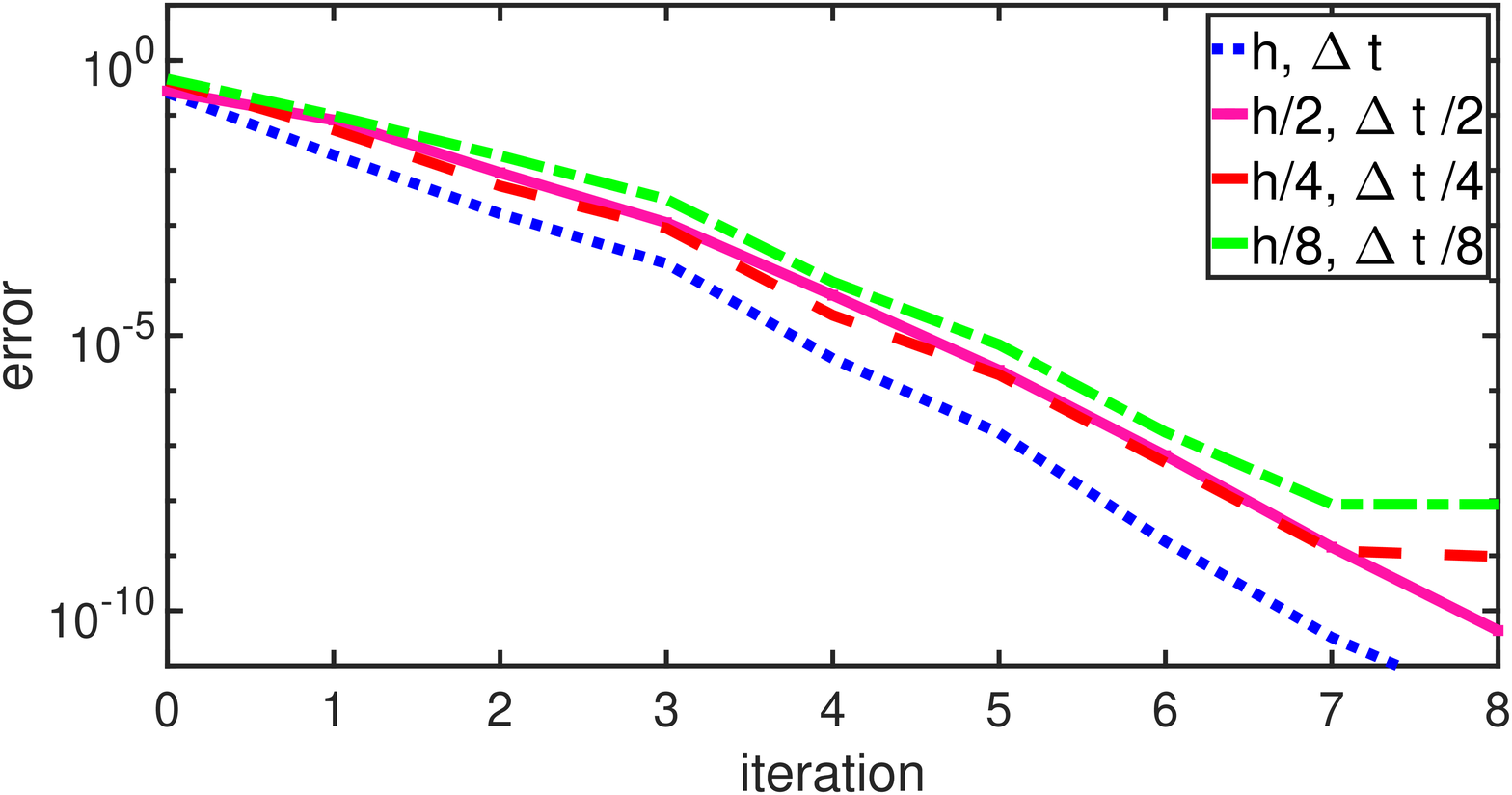} }}
    \caption{PA-I: On the left convergence for different $\Delta T$; On the right convergence for different $h, \Delta t$.}
    \label{diff_dtdx}
\end{figure}
We plot the error curves on the left panel in Figure \ref{diff_T_N_eps} for short as well as long time window with $\epsilon=0.0725, J=200$ and $h=1/64$. The method converges in four iterations to the fine solution of temporal accuracy $O(10^{-4})$ for different $T$. By ignoring the computational cost of the coarse operator, we can see that the Parareal method is 40 times faster than serial method on a single processor for $T=8$. It is evident from the left plot of Figure \ref{diff_T_N_eps} that one can achieve more speed up by including more processors ($N$). To see the dependency on the parameter $\epsilon$, we plot the error curves on the right panel in Figure \ref{diff_T_N_eps} for different $\epsilon$ by taking $T=1, N=50, J=200$. We observe that the method is almost immune to the choice of $\epsilon$.
\begin{figure}
    \centering
    \subfloat{{\includegraphics[height=3.5cm,width=6cm]{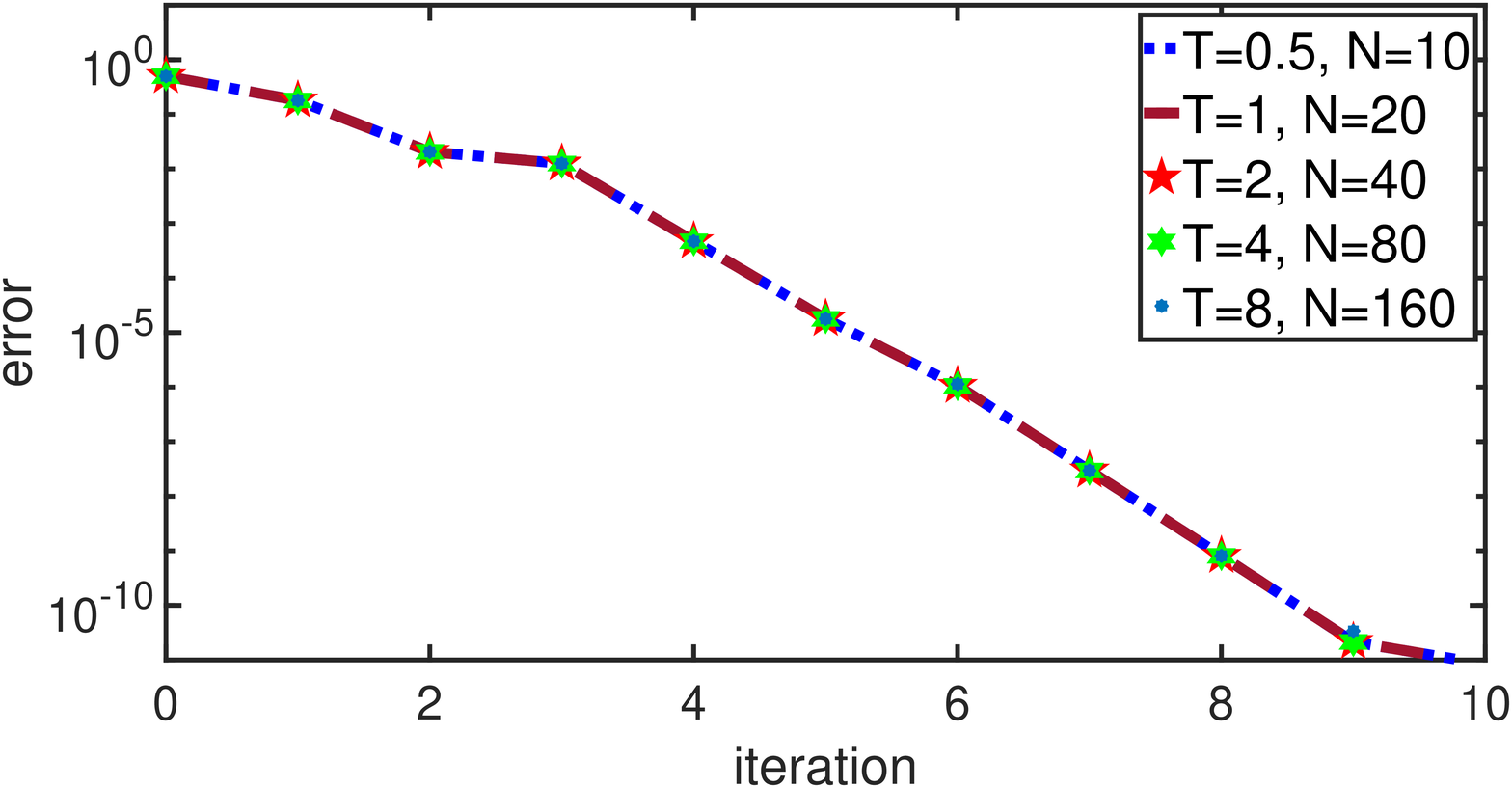} }}
     \subfloat{{\includegraphics[height=3.5cm,width=6cm]{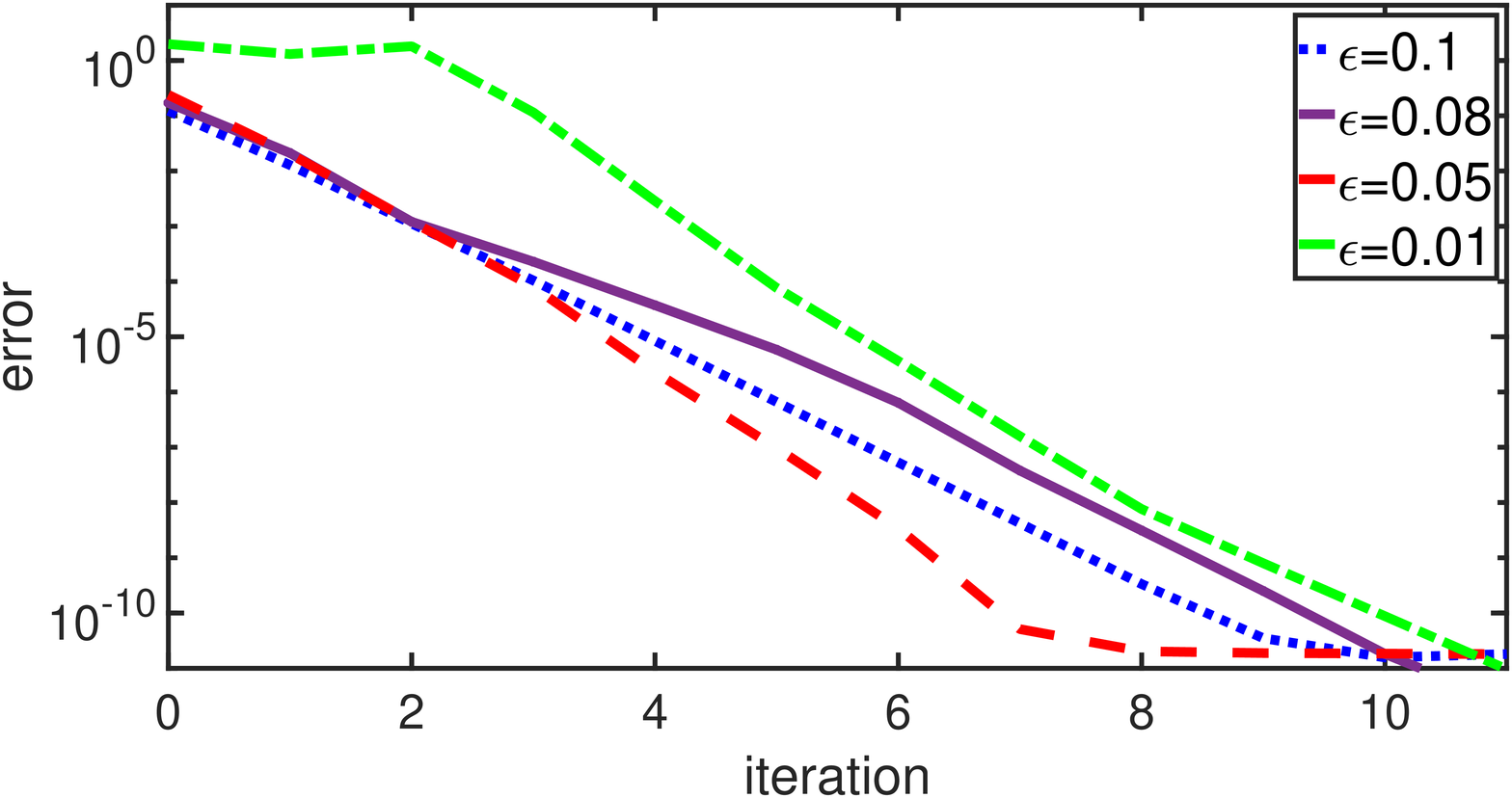} }}
    \caption{PA-I: On the left convergence for different $T, N$, and on the right convergence for different $\epsilon$.}
    \label{diff_T_N_eps}
\end{figure}

To perform the numerical experiments in 2D we take the discretization parameter $h=1/32$ on both direction. We plot the comparison of error contraction on the left panel in Figure \ref{diff_T_N_bound} for $T=1, N=20, J=200$ and $\epsilon=0.0725$. We plot the error curves on the right in Figure \ref{diff_T_N_bound} for short as well as long time window with $\epsilon=0.0725, J=200$. The method converges in four iterations to the fine solution of temporal accuracy $O(10^{-4})$ for different $T$. We observe similar convergence behaviour of PA-I in 2D as in 1D with respect to different situation and so we skip those experiments here.
\begin{figure}
    \centering
    \subfloat{{\includegraphics[height=3.5cm,width=6cm]{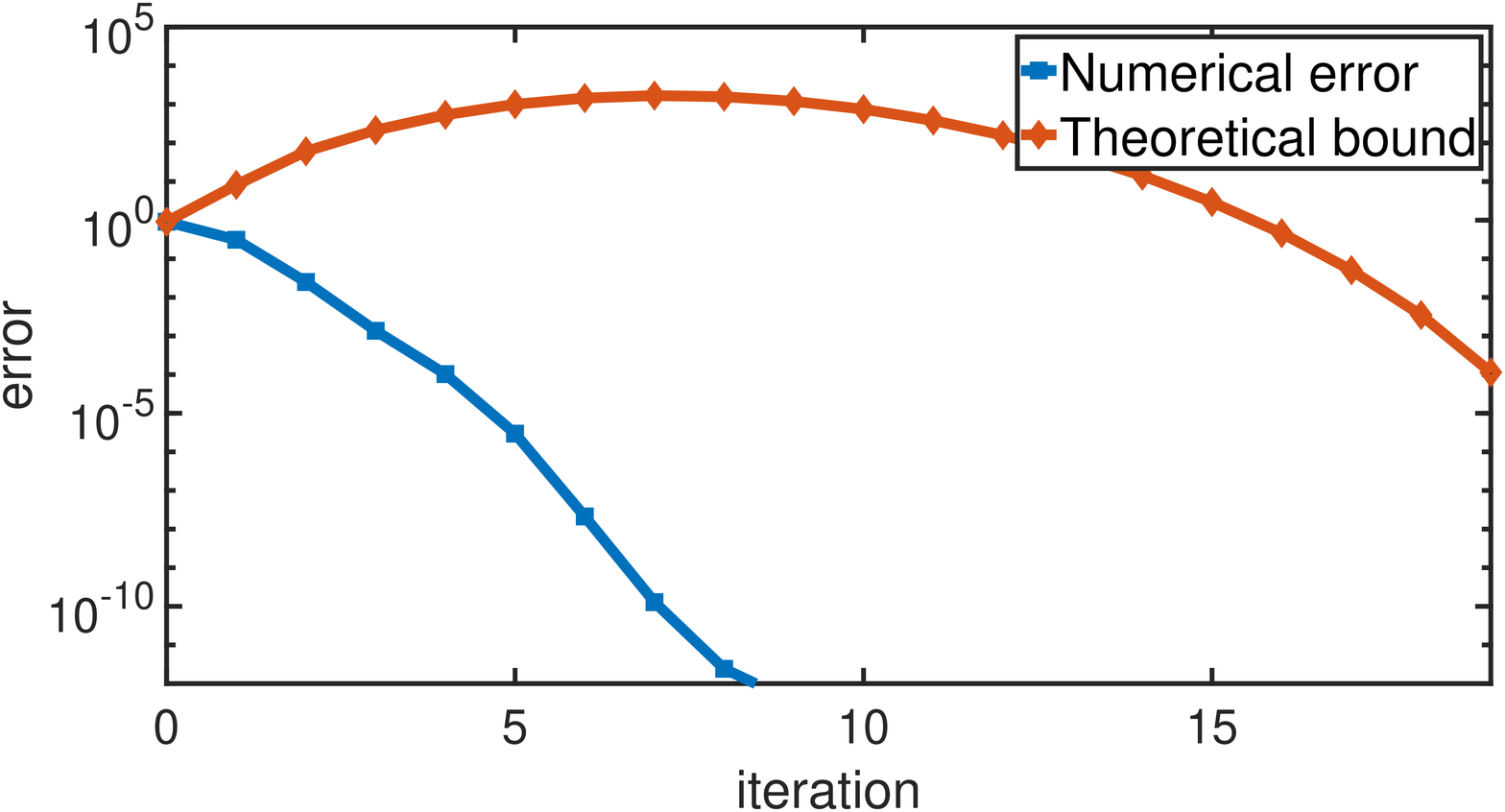} }}
     \subfloat{{\includegraphics[height=3.5cm,width=6cm]{diff_T_N_pa1} }}
    \caption{PA-I: On the left comparison of theoretical and numerical error, and on the right convergence for different $T, N$.}
    \label{diff_T_N_bound}
\end{figure}

%% file: numerics_paII.tex
\subsection{Numerical Experiments of PA-II}
1D case: The comparison of numerical error and theoretical estimate from Theorem \ref{thm4} can be seen from the left plot in Figure \ref{diff_errbound_pa2} for $T=1, h=1/64, N=20, J=200$ and $\epsilon=0.0725$. On the right we plot the error curves for more refined solution for $T=1, \epsilon=0.0725, h=1/64, \Delta t=1/200$. We can see that the convergence is independent of mesh parameters.
\begin{figure}[h]
    \centering
    \subfloat{{\includegraphics[height=3.5cm,width=6cm]{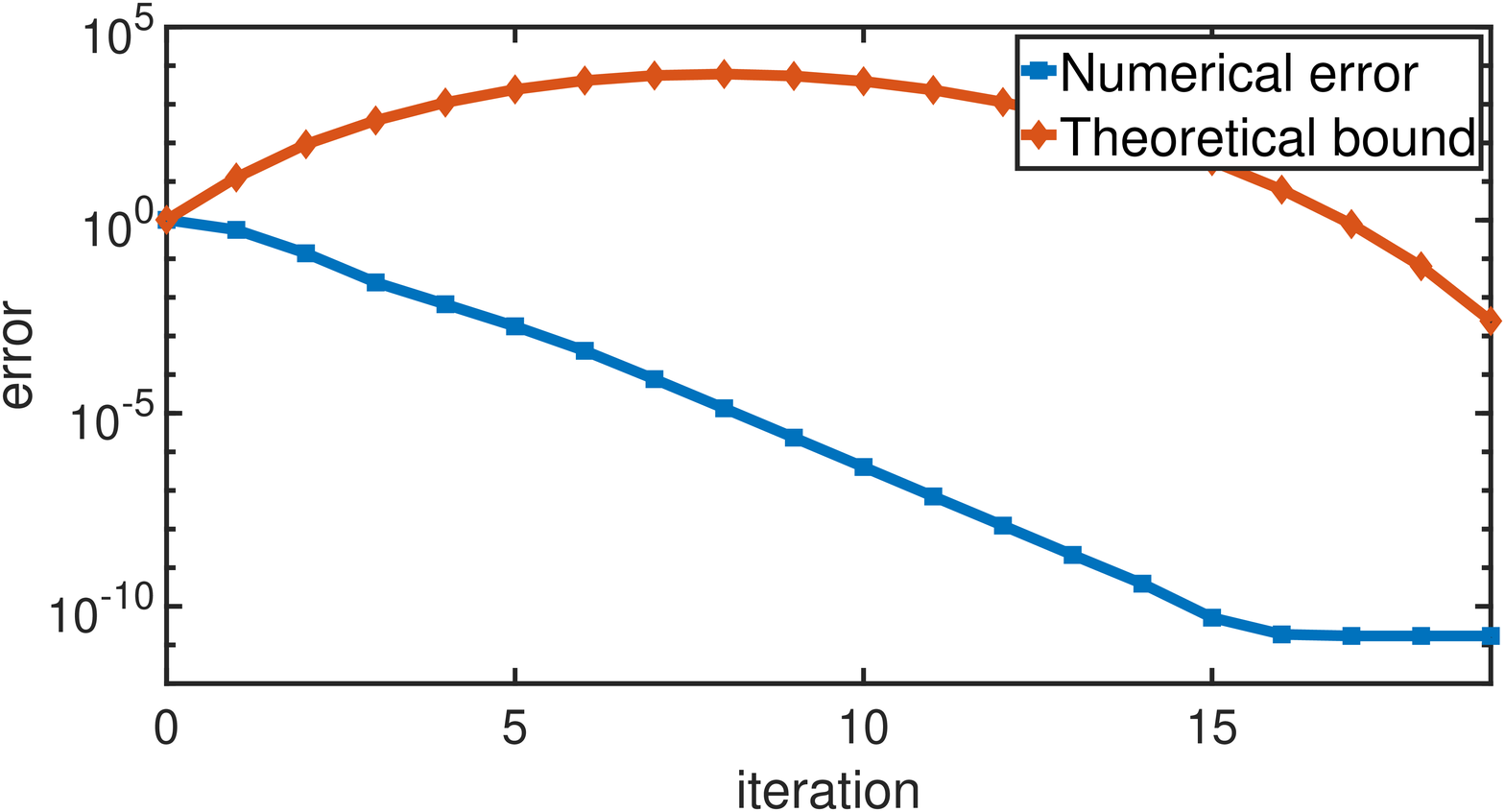} }}
     \subfloat{{\includegraphics[height=3.5cm,width=6cm]{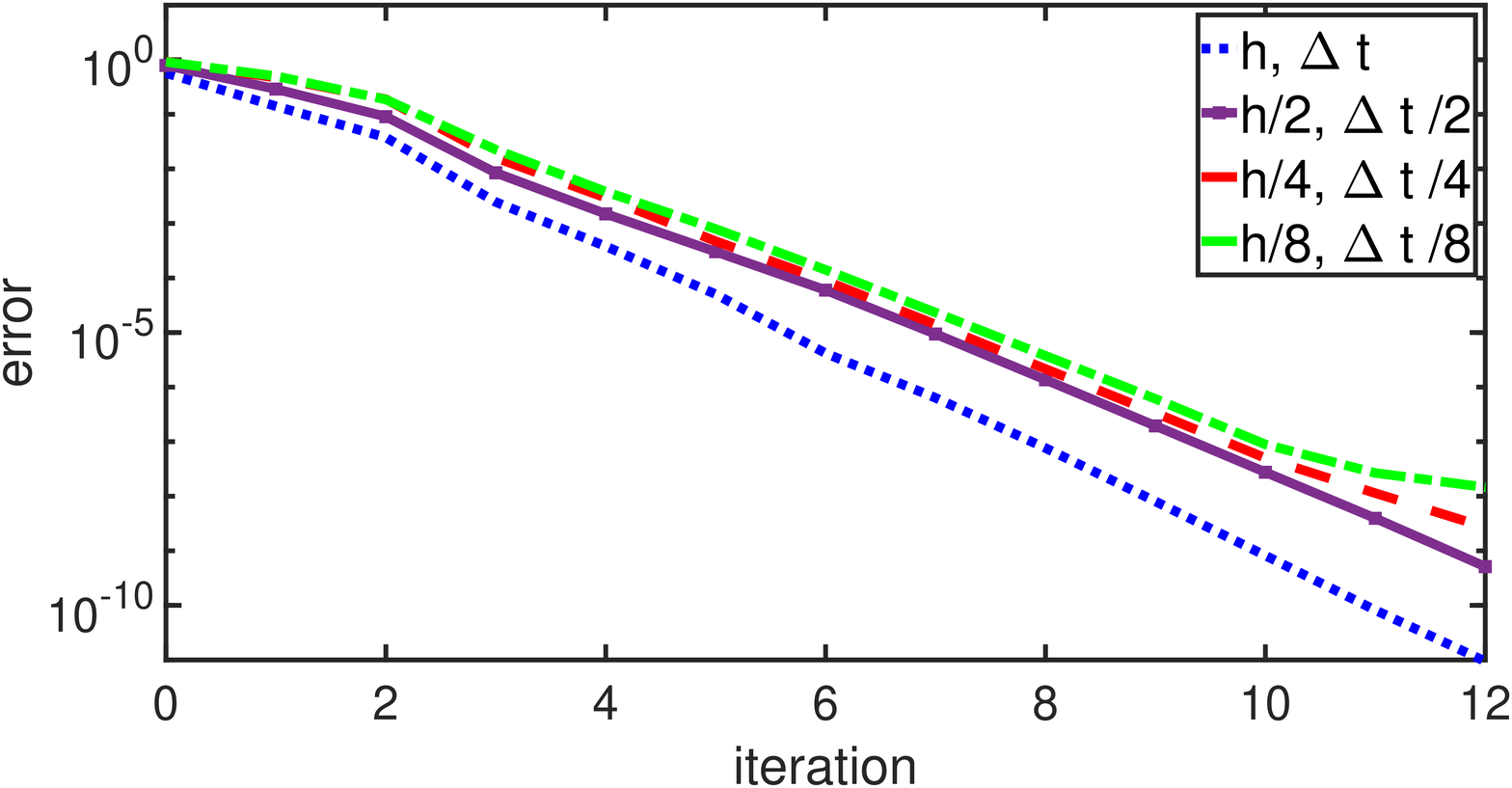} }}
    \caption{PA-II: On the left comparison of theoretical and numerical error, and on the right convergence for different mesh sizes.}
    \label{diff_errbound_pa2}
\end{figure}
We plot the error curves on the left in Figure \ref{diff_TN_pa2} for short as well as long time window with $\epsilon=0.0725, J=200$ and $h=1/64$. We can see that one get the speed up compared to serial solve. To see the dependency on the parameter $\epsilon$, we plot the error curves on the right in Figure \ref{diff_TN_pa2} for different $\epsilon$ by taking $T=1, N=50, J=200$. We can see that the PA-II is sensitive to the choice of $\epsilon$, namely for the very small $\epsilon$.
\begin{figure}
    \centering
    \subfloat{{\includegraphics[height=3.5cm,width=6cm]{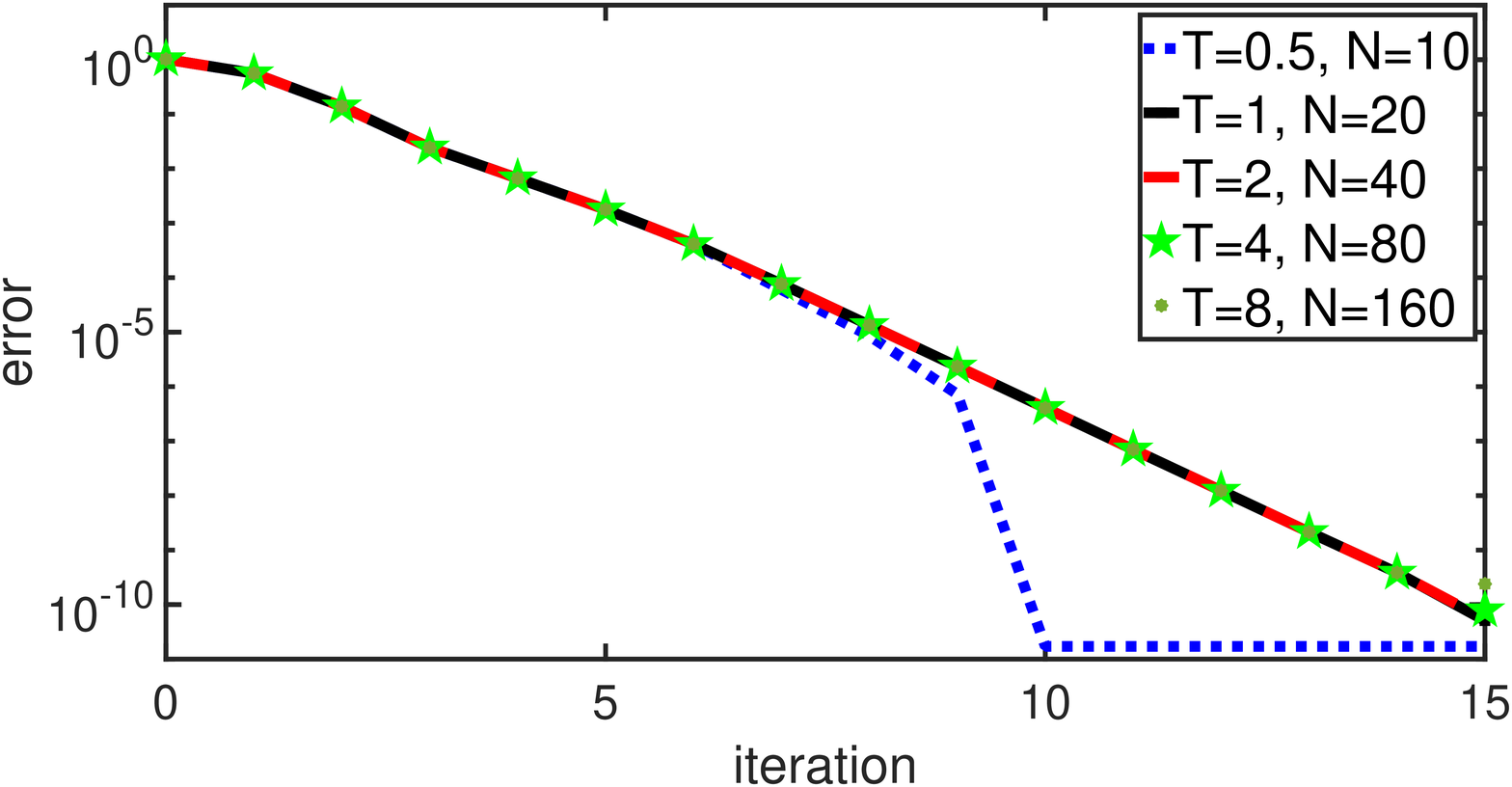} }}
     \subfloat{{\includegraphics[height=3.5cm,width=6cm]{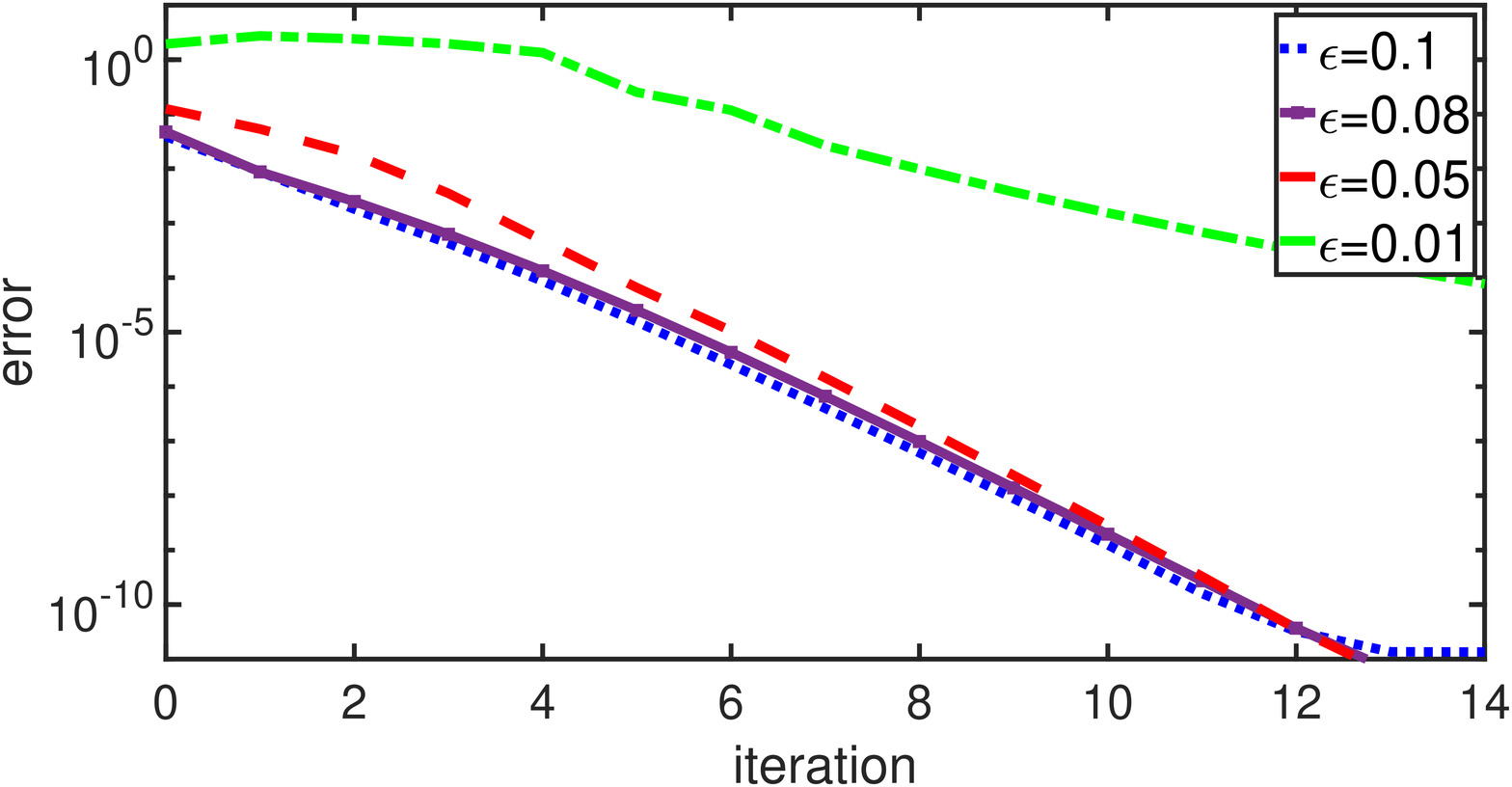} }}
    \caption{PA-II: On the left convergence for different $T, N$, and on the right $\epsilon$ dependency on the convergence.}
    \label{diff_TN_pa2}
\end{figure}

 2D case: We take the same discretization parameter $h=1/32$ on both direction. We plot the comparison of error contraction on the left panel in Figure \ref{diff_T_N_bound_pa2} for $T=1, N=20, J=200$ and $\epsilon=0.0825$. We plot the error curves on the right in Figure \ref{diff_T_N_bound_pa2} for short as well as long time window with $\epsilon=0.0725, J=200$. We observe similar convergence behaviour of PA-II in 2D as in 1D with respect to different situation.
\begin{figure}
    \centering
    \subfloat{{\includegraphics[height=3.5cm,width=6cm]{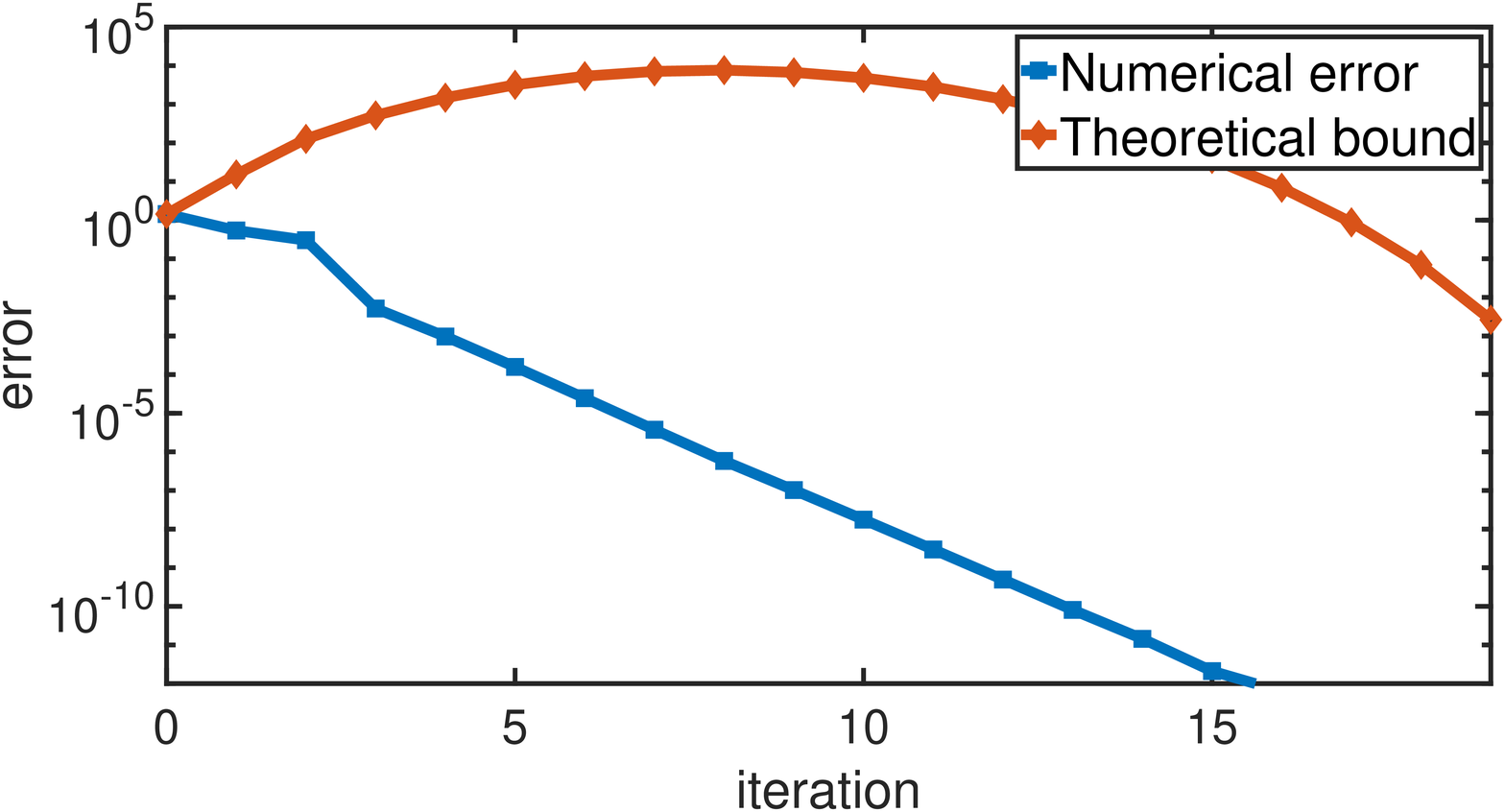} }}
     \subfloat{{\includegraphics[height=3.5cm,width=6cm]{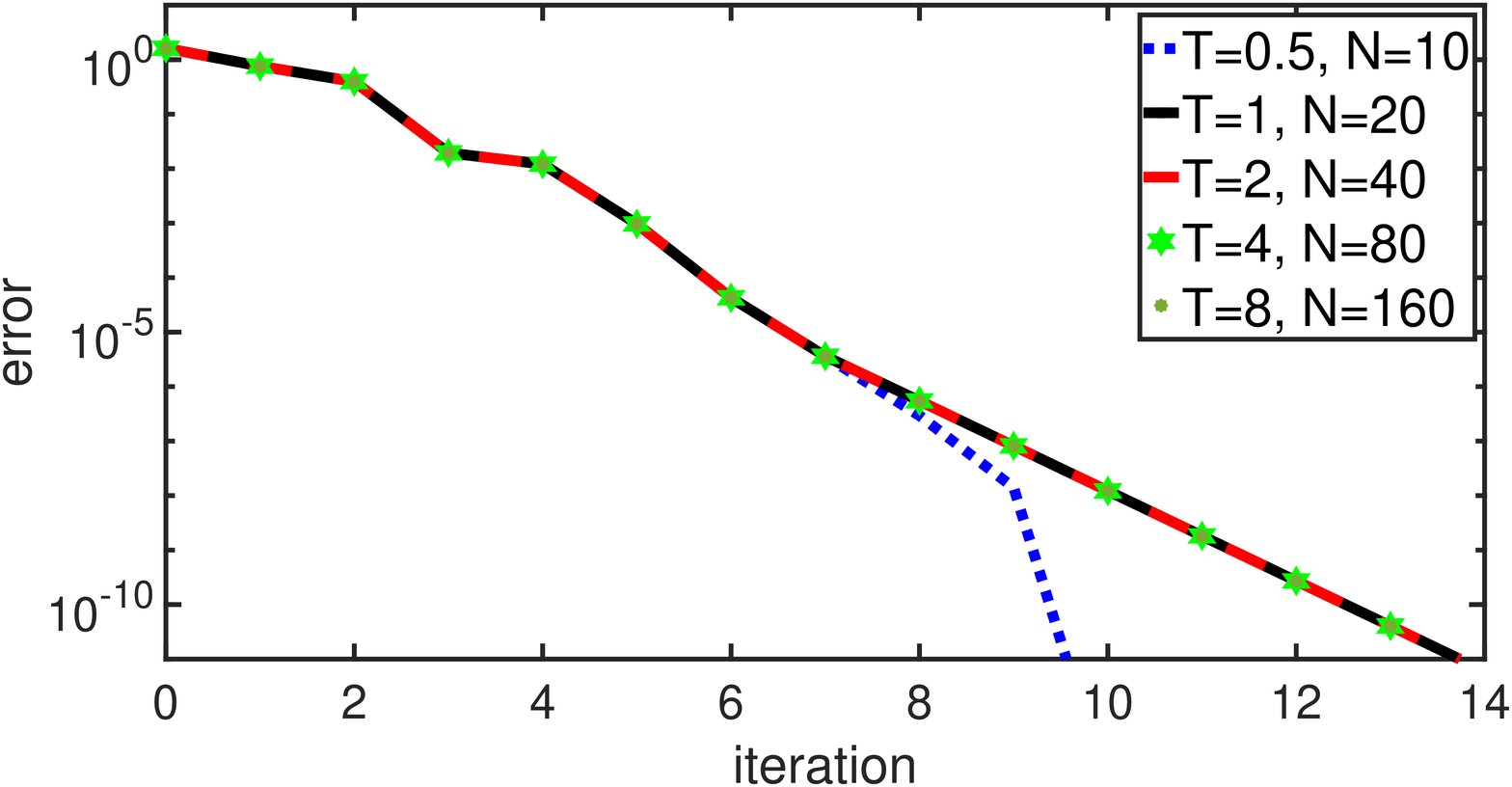} }}
    \caption{PA-II: On the left comparison of theoretical and numerical error, and on the right convergence for different $T, N$.}
    \label{diff_T_N_bound_pa2}
\end{figure}

%% file: numerics_paIII.tex
\subsection{Numerical Experiments of PA-III}
1D case: The comparison of numerical error and theoretical estimates from Theorem \ref{thm6} can be seen in the left plot of Figure \ref{diff_errbound_pa3} for $T=1, h=1/64, N=20, J=200$ and $\epsilon=0.0725$. On the right panel we plot the error curves for more refined solution for $T=1, \epsilon=0.0725, h=1/64, \Delta t=1/200$. We can see that convergence is independent of mesh parameters.
\begin{figure}[h]
    \centering
    \subfloat{{\includegraphics[height=3.5cm,width=6cm]{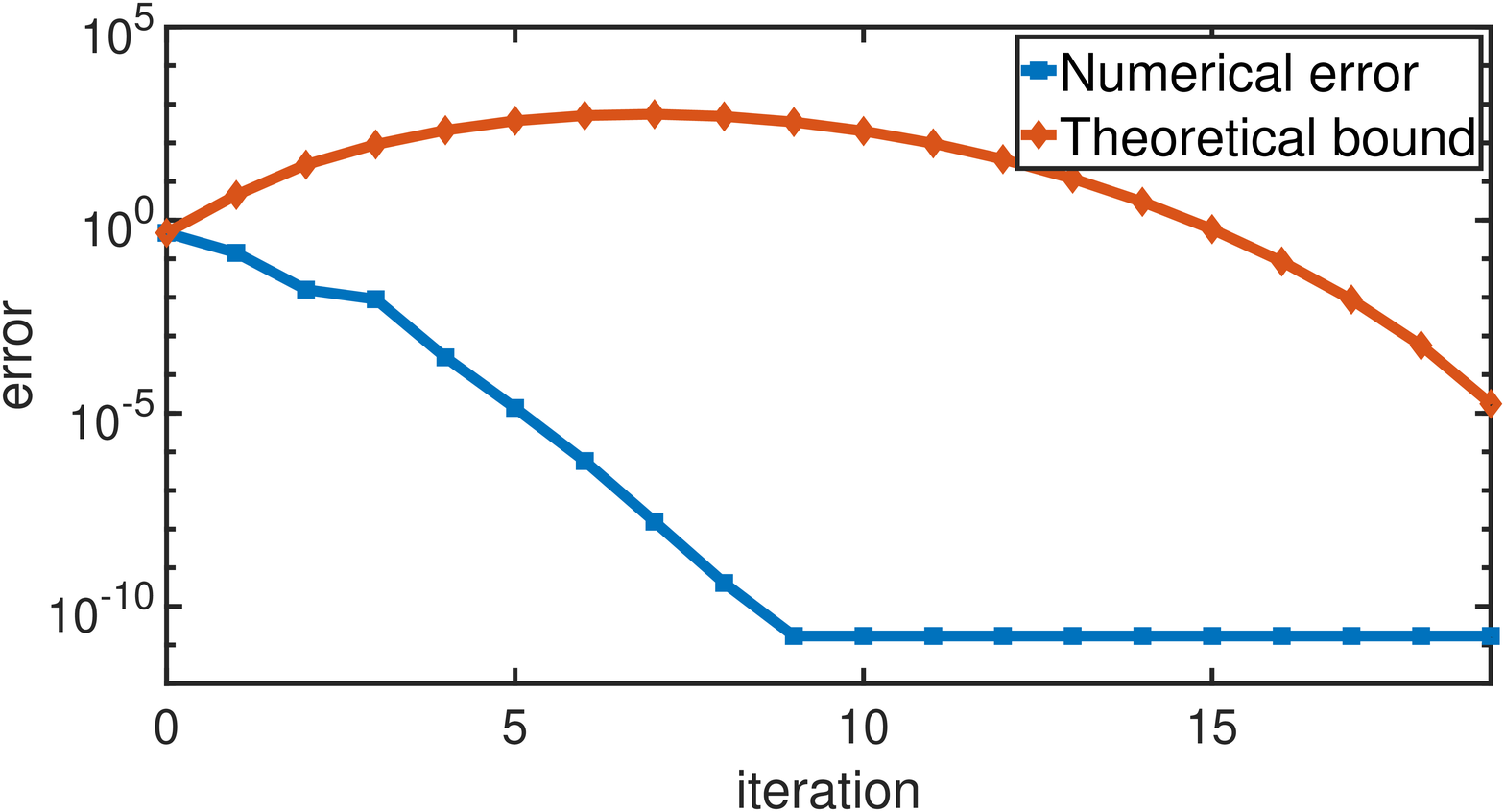} }}
     \subfloat{{\includegraphics[height=3.5cm,width=6cm]{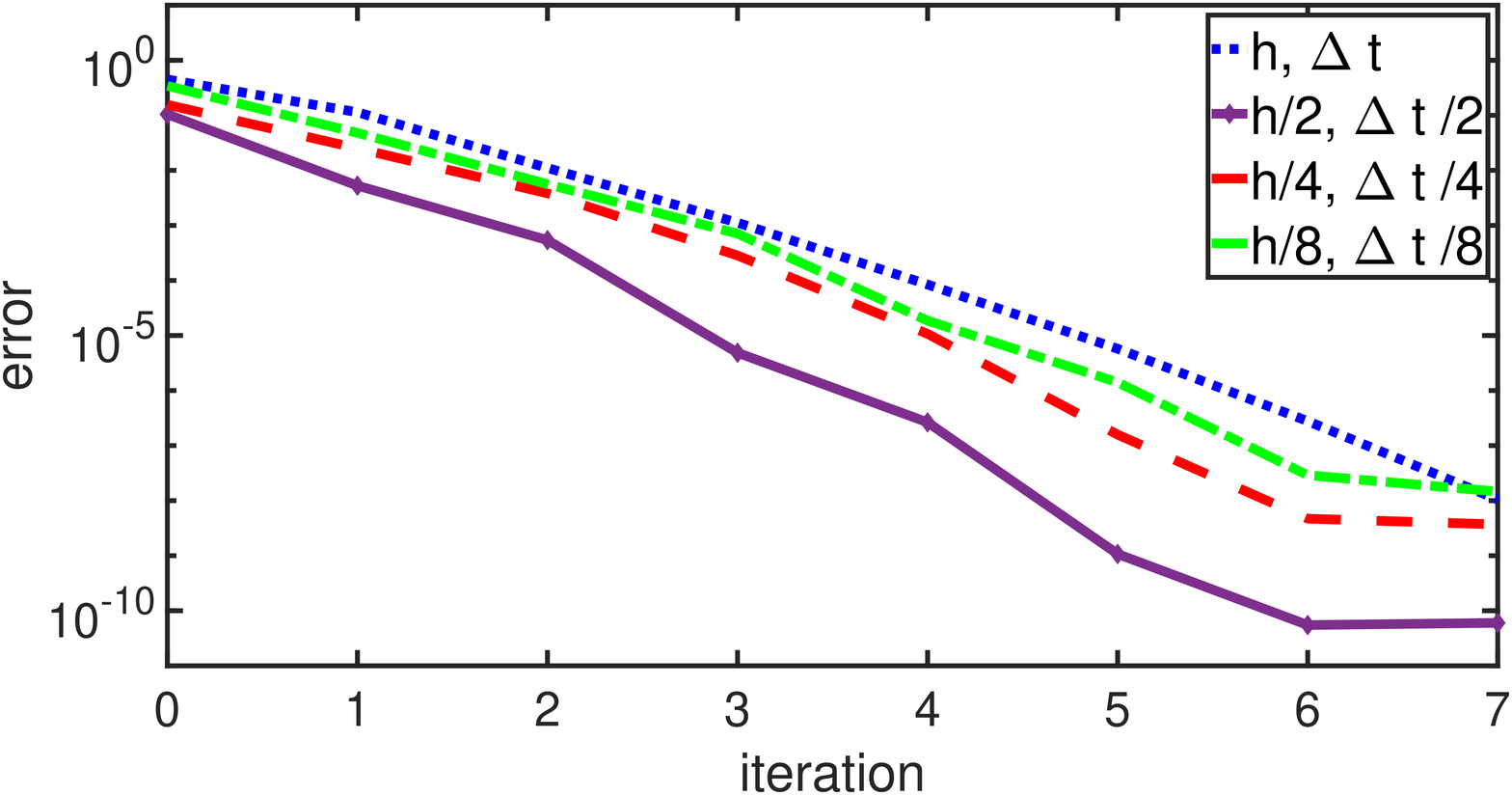} }}
    \caption{PA-III: On the left comparison of theoretical and numerical error, and on the right convergence for different mesh sizes.}
    \label{diff_errbound_pa3}
\end{figure}
We plot the error curves on the left in Figure \ref{diff_TN_pa3} for short as well as long time window with $\epsilon=0.0725, J=200$ and $h=1/64$. The method converges in four iteration to the fine resolution of temporal accuracy $O(10^{-4})$ for different $T$ and one get the speed up compared to sequential solve. To see the dependency on the parameter $\epsilon$, we plot the error curves on the right in Figure \ref{diff_TN_pa3} for different $\epsilon$ by taking $T=1, N=50, J=200$. We can see that the convergence of PA-III is independent of the choice of $\epsilon$. As PA-II and PA-III converge to the fine solution given by \eqref{approx3}, we can compare them. Since PA-II is sensitive towards small $\epsilon$, therefore PA-III is the best choice to approximate fine solution given by \eqref{approx3}.
\begin{figure}
    \centering
    \subfloat{{\includegraphics[height=3.5cm,width=6cm]{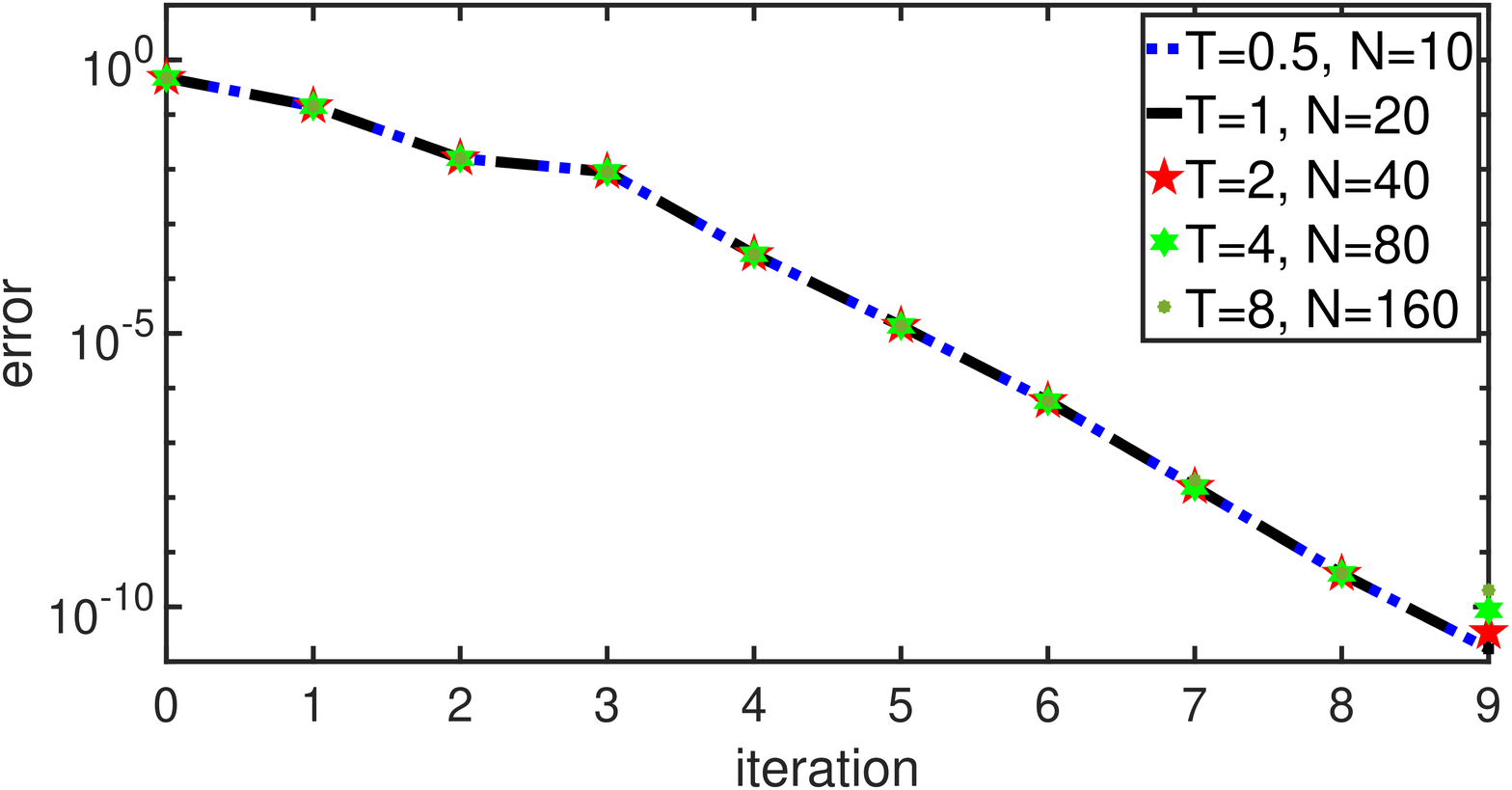} }}
     \subfloat{{\includegraphics[height=3.5cm,width=6cm]{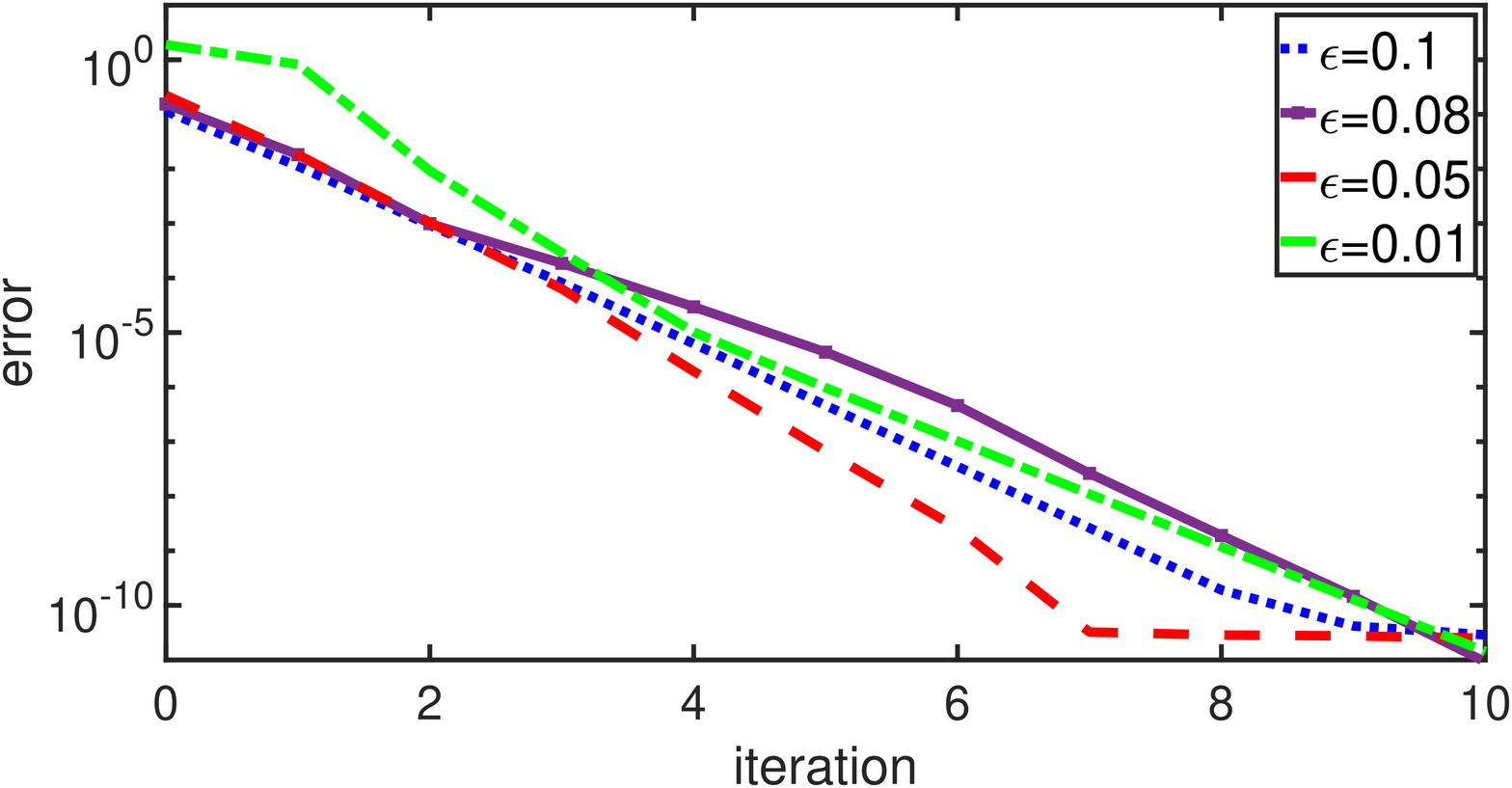} }}
    \caption{PA-III: On the left convergence for different $T, N$, and on the right $\epsilon$ dependency on the convergence.}
    \label{diff_TN_pa3}
\end{figure}

 2D case: We take the discretization parameter $h=1/32$ on both direction and plot the comparison of error contraction on the left in Figure \ref{diff_T_N_bound_pa3} for $T=1, N=20, J=200$ and $\epsilon=0.0625$. We plot the error curves on the right for short as well as long time window with $\epsilon=0.0725, J=200$.  We observe similar convergence behaviour of PA-III in 2D as in 1D for different situation.
\begin{figure}
    \centering
    \subfloat{{\includegraphics[height=3.5cm,width=6cm]{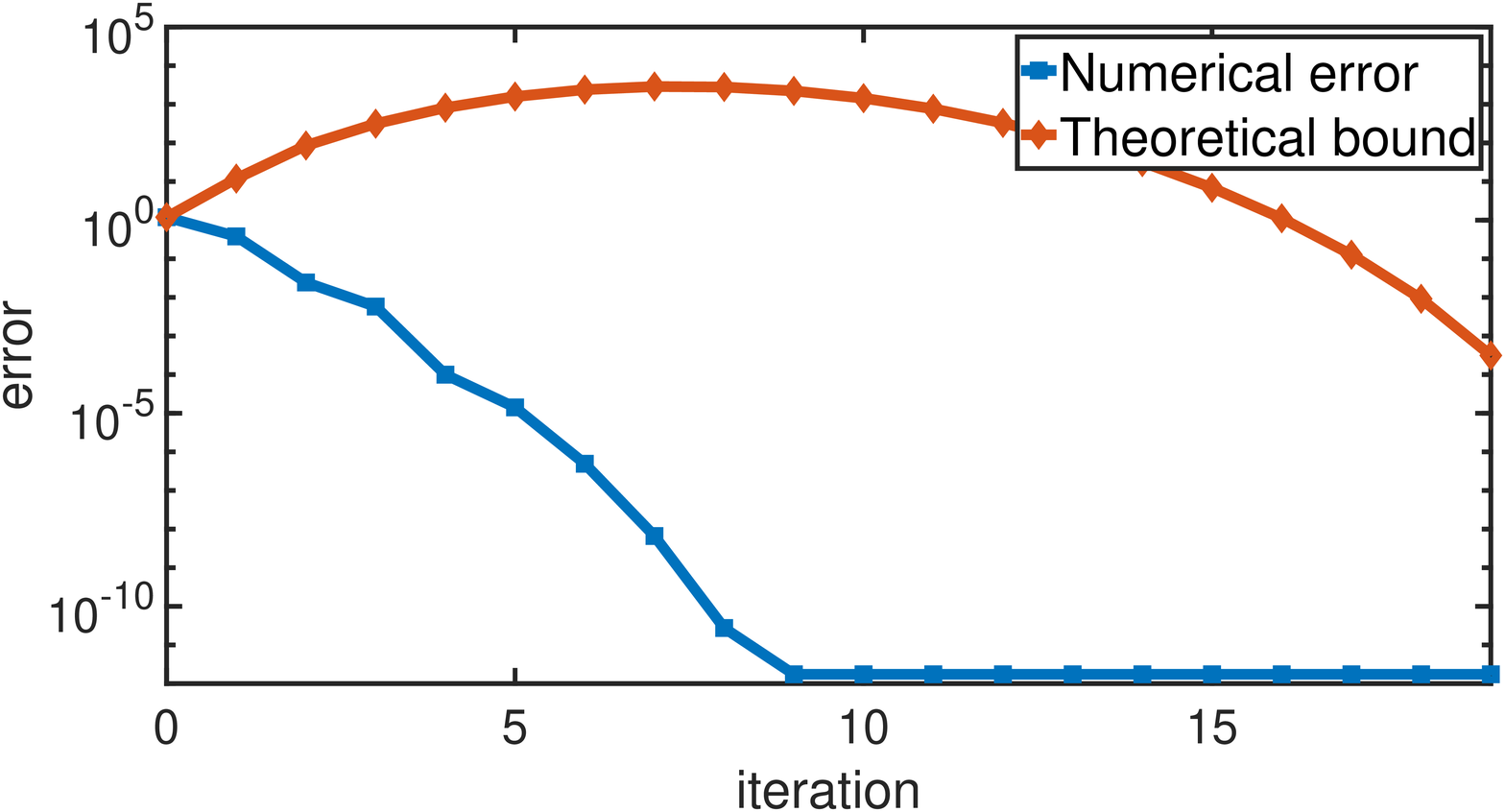} }}
     \subfloat{{\includegraphics[height=3.5cm,width=6cm]{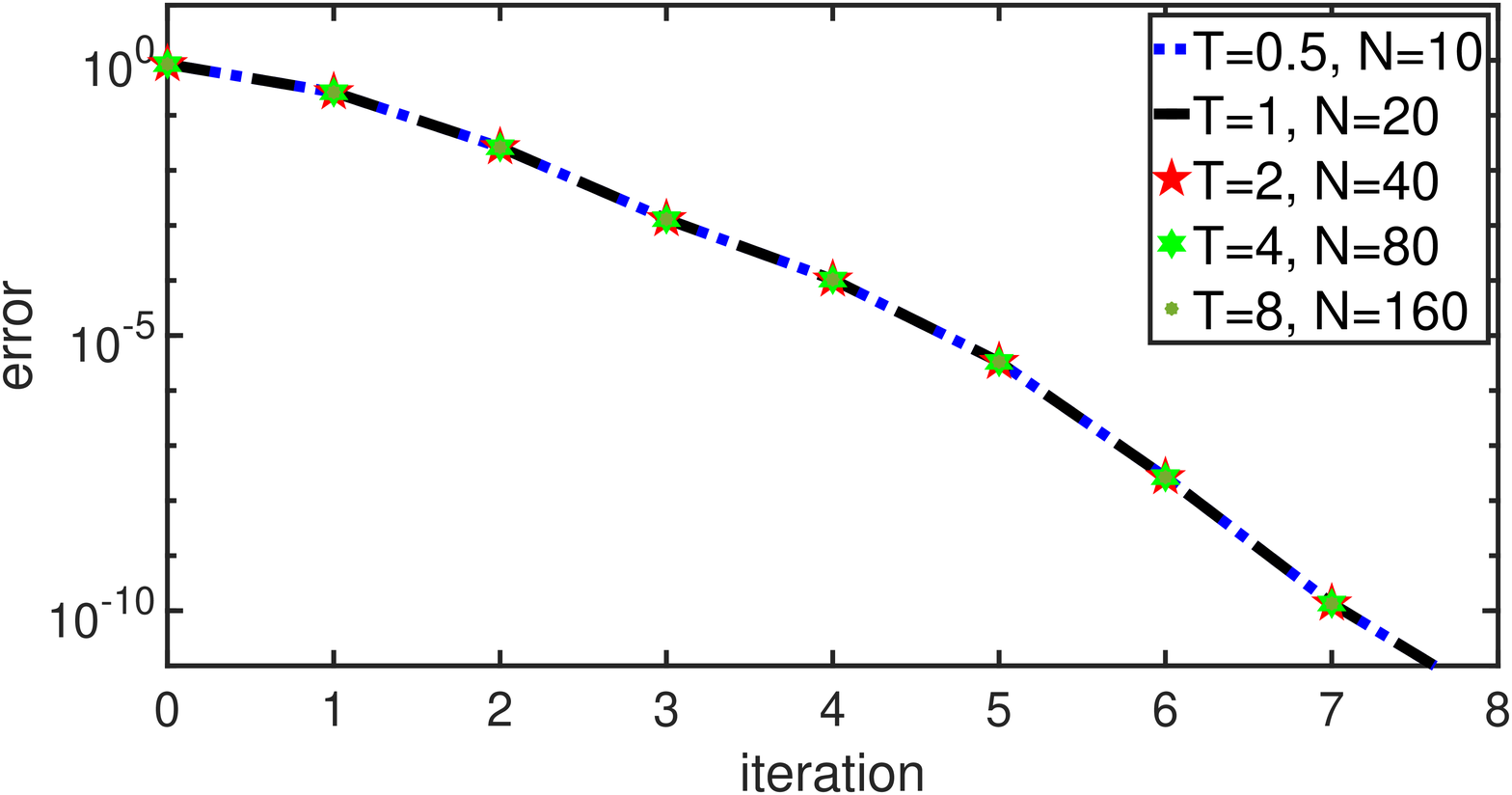} }}
    \caption{PA-III: On the left comparison of theoretical and numerical error, and on the right convergence for different $T, N$.}
    \label{diff_T_N_bound_pa3}
\end{figure}

%% file: numerics_npaI.tex
\subsection{Numerical Experiments of NPA-I}
1D case: The nonlinear fine propagator is obtained using the Newton method with a tolerance $1e {-10}$. To implement the  theoretical bound prescribed in Theorem \ref{thm_npa1} we have to estimate the quantity $C_1$ numerically, which depends on choice of $\epsilon, J, \Delta T, u^0$. The comparison of numerical error and theoretical estimates can be seen on the left plot of Figure \ref{diff_errbound_npa1} for $\Delta T=1, h=1/64, N=20, J=200, \epsilon=0.0725$ and $C_1=0.1181$. On the right panel in Figure \ref{diff_errbound_npa1} we plot the error curves for more refined solution for $T=1, \epsilon=0.0725, h=1/64, \Delta t=1/200$. We can see that convergence is independent of mesh parameters.
\begin{figure}[h]
    \centering
    \subfloat{{\includegraphics[height=3.5cm,width=6cm]{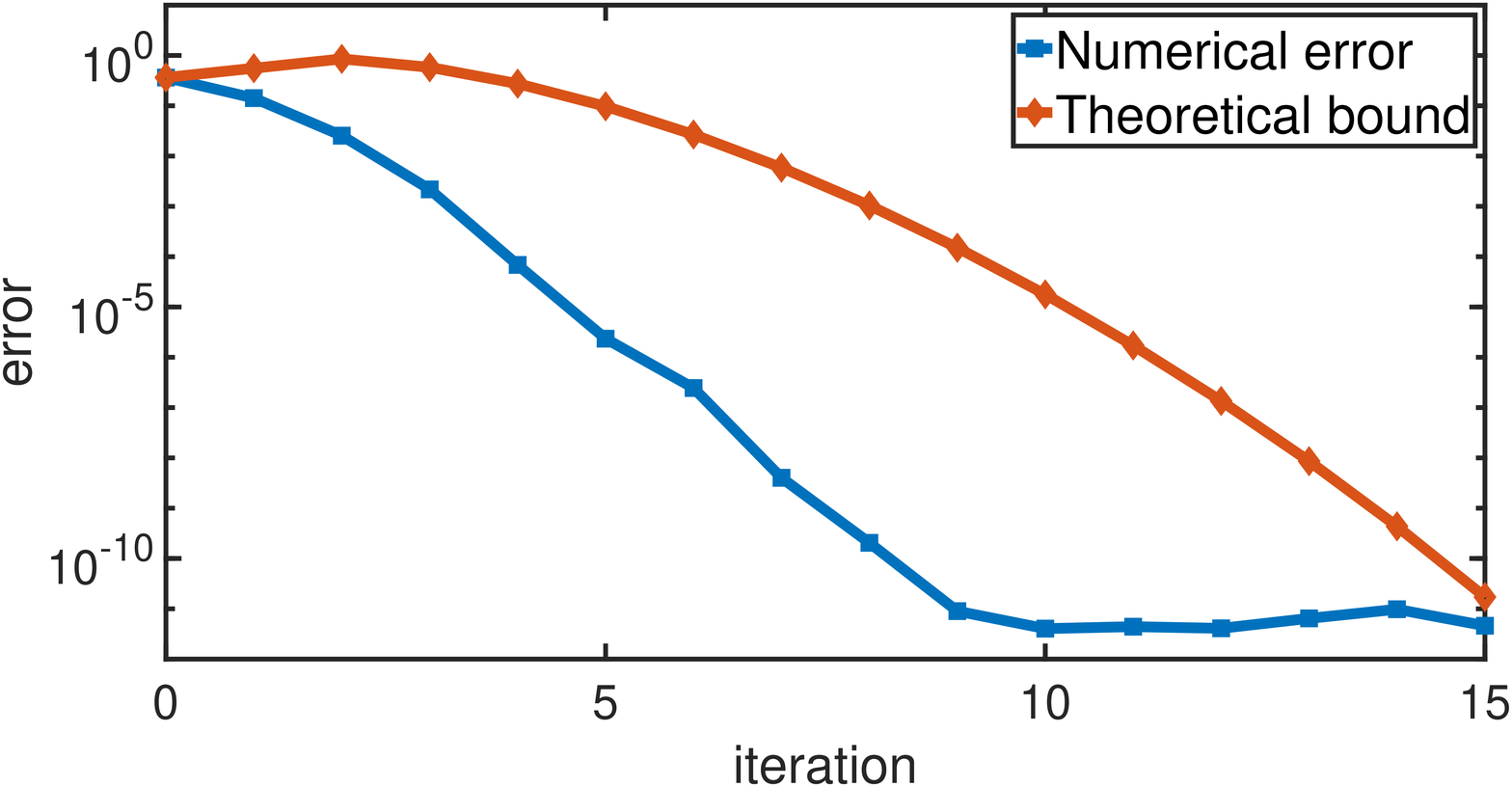} }}
     \subfloat{{\includegraphics[height=3.5cm,width=6cm]{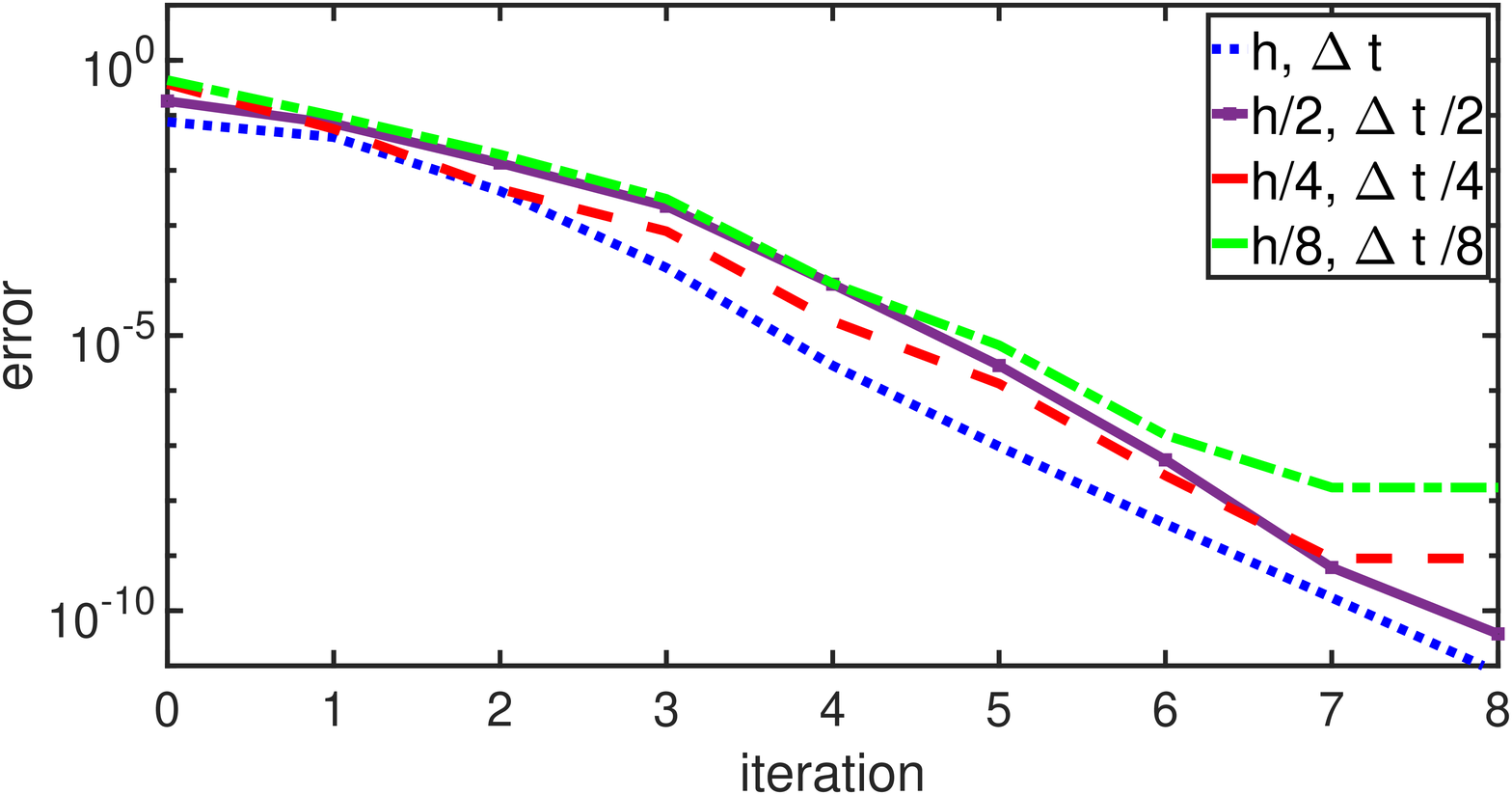} }}
    \caption{NPA-I: On the left comparison of theoretical and numerical error, and on the right convergence of refined solution.}
    \label{diff_errbound_npa1}
\end{figure}
We plot the error curves on the left in Figure \ref{diff_TN_npa1} for short as well as long time window with $\epsilon=0.0725, J=200$ and $h=1/64$. To see the dependency on the parameter $\epsilon$, we plot the error curves on the right panel in Figure \ref{diff_TN_npa1} for different $\epsilon$ by taking $T=1, N=50, J=200$. We can see that the NPA-I is independent of the choice of $\epsilon$. 
\begin{figure}
    \centering
    \subfloat{{\includegraphics[height=3.5cm,width=6cm]{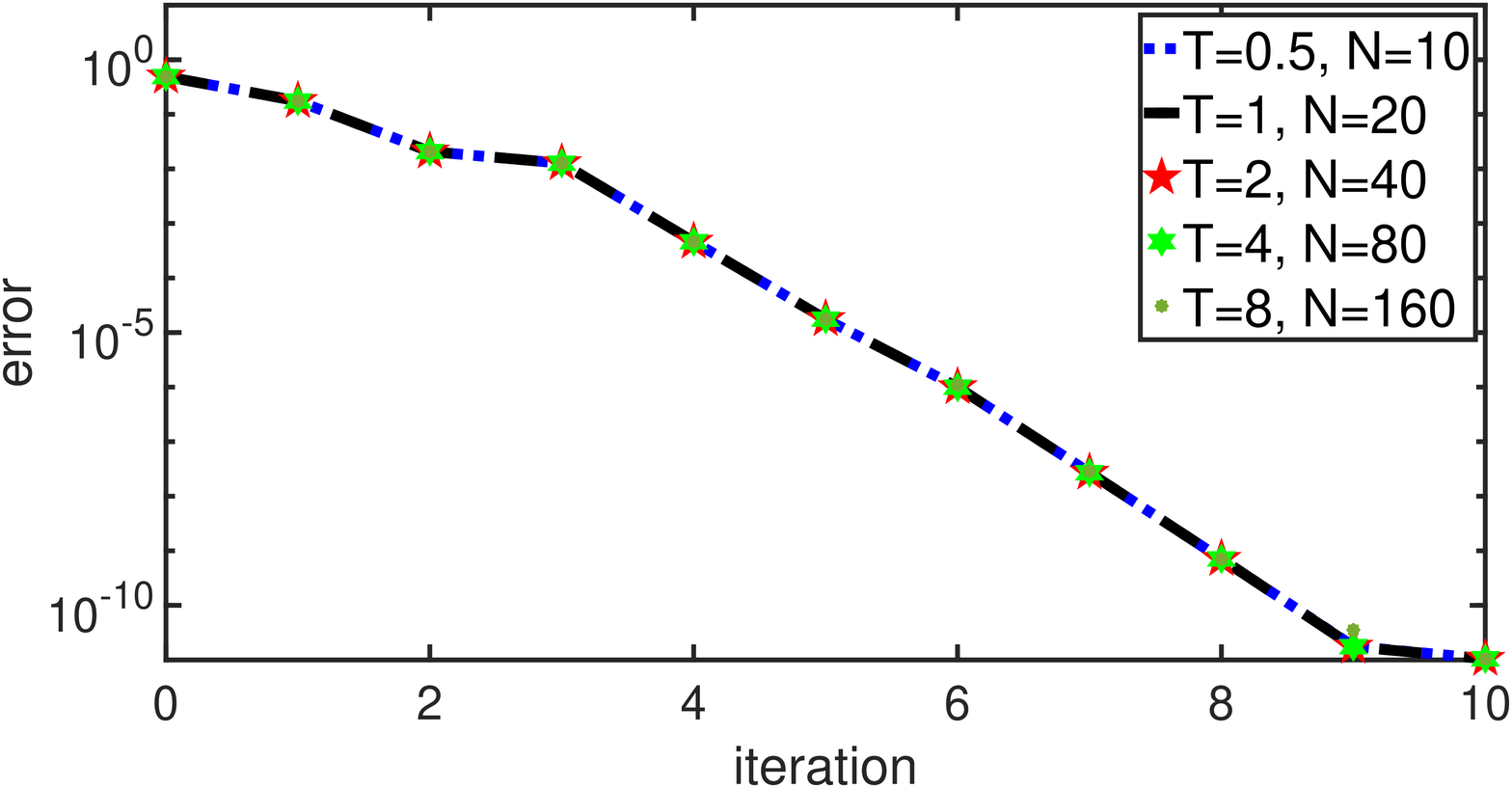} }}
     \subfloat{{\includegraphics[height=3.5cm,width=6cm]{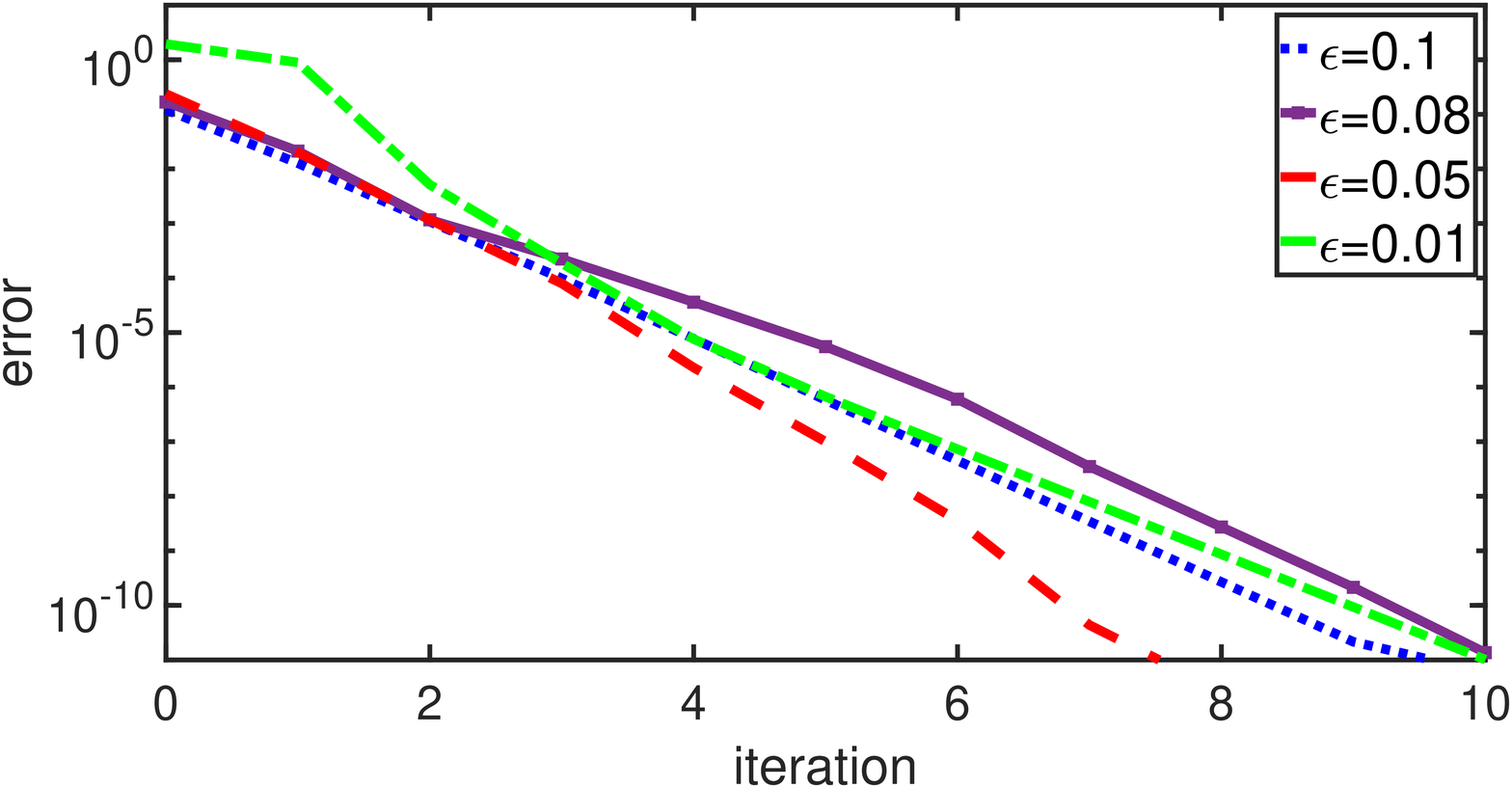} }}
    \caption{NPA-I: On the left convergence for different $T, N$, and on the right $\epsilon$ dependency on convergence.}
    \label{diff_TN_npa1}
\end{figure}
On the left panel in Figure \ref{diff_T_N_bound_npa1} we plot the error  curves with respect to different number of time slices for $T=50, \epsilon=0.0725, h=1/64, J=150$. We can see that convergence is independent of time decomposition. One can observe that a speed up of $80$ times compared to serial solve for $N=400$.

2D case: We take the discretization parameter $h=1/32$ on both direction. As we observe similar convergence behaviour in 2D as in 1D, we only plot the error curves on the right in Figure \ref{diff_T_N_bound_npa1} for short as well as long time window with $\epsilon=0.0725, J=200$. We omit the other experiments in 2D as we observe similar convergence behaviour as in 1D.
\begin{figure}
    \centering
    \subfloat{{\includegraphics[height=3.5cm,width=6cm]{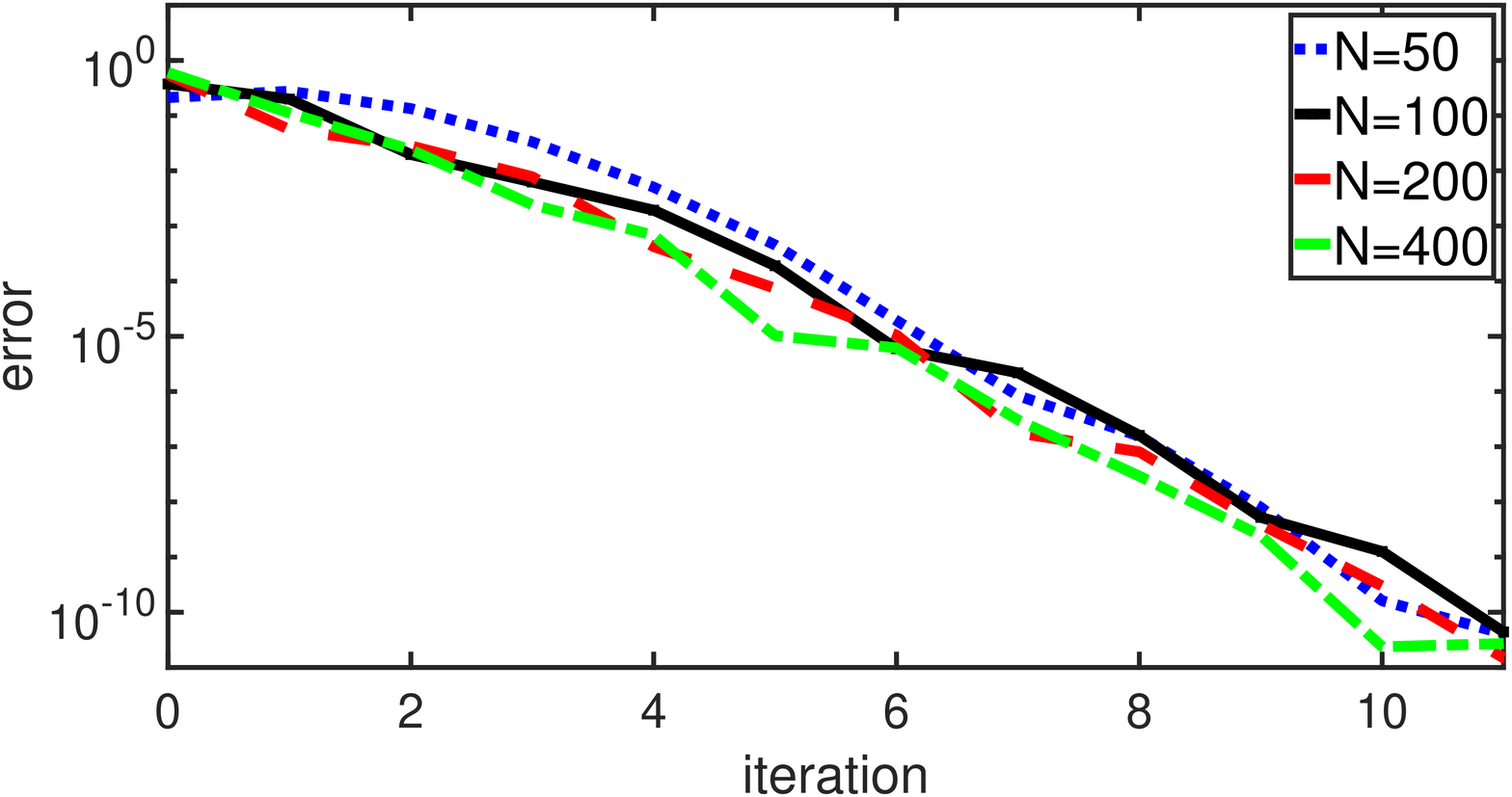} }}
     \subfloat{{\includegraphics[height=3.5cm,width=6cm]{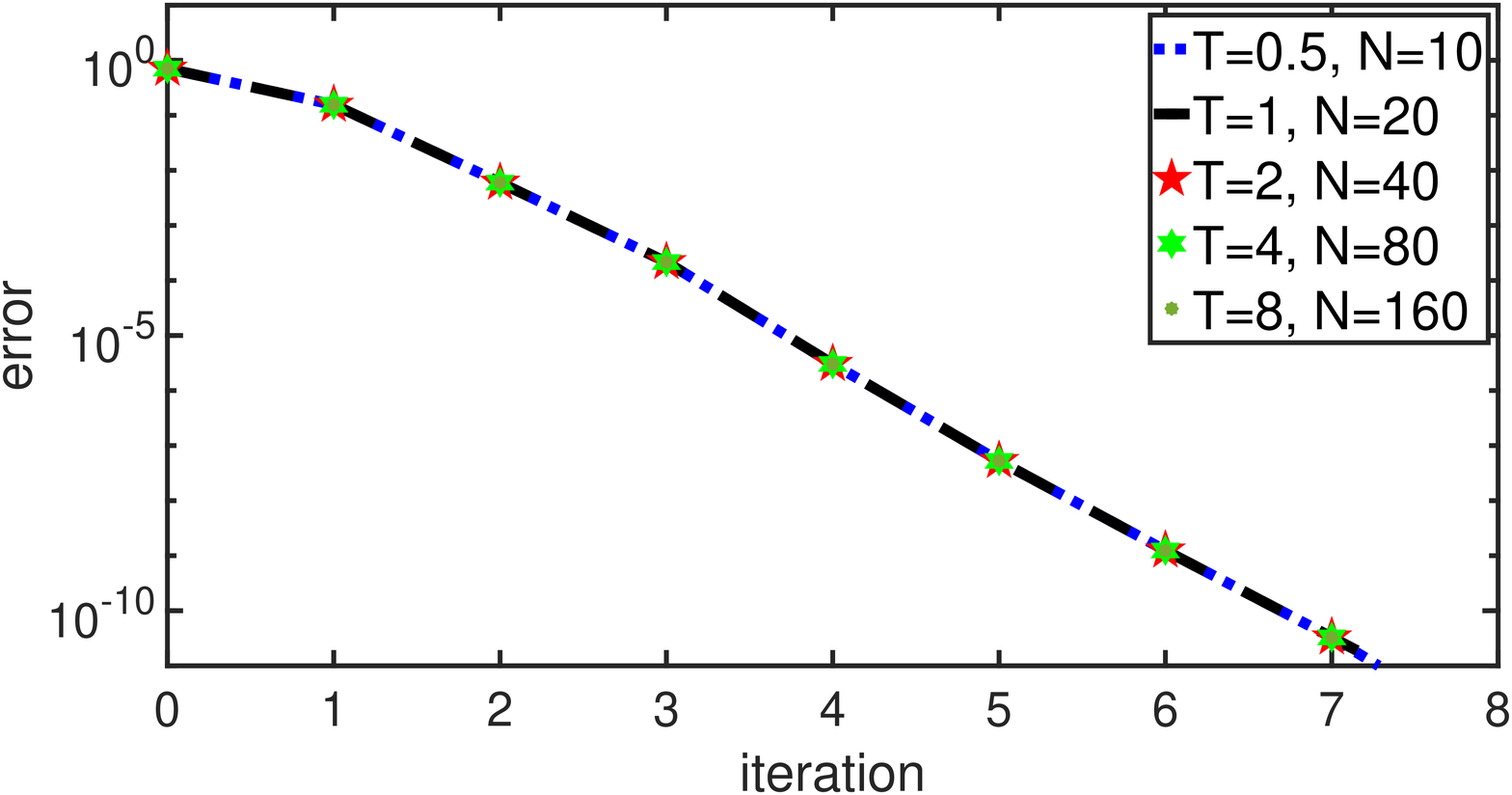} }}
    \caption{NPA-I: On the left convergence for different $N$, and on the right convergence for different $T, N$ in 2D.}
    \label{diff_T_N_bound_npa1}
\end{figure}

%% file: numerics_npaII.tex
\subsection{Numerical Experiments of NPA-II}
1D case: In this case we have nonlinear solvers for both fine and coarse propagator by the Newton method with a tolerance $1e {-10}$. To implement the theoretical bound prescribed in Theorem \ref{thm_npa2} we have to estimate the quantity $C_1$ numerically, which depends on the choice of $\epsilon, J, \Delta T, u^0$. The comparison of numerical error and theoretical estimate can be seen on the left plot of Figure \ref{diff_errbound_npa2} for $\Delta T=1, h=1/64, N=20, J=200, \epsilon=0.0725$ and $C_1=0.1498$. On the right we plot the error curves for more refined solution for $T=1, \epsilon=0.0725, h=1/64, \Delta t=1/200$. We can see that convergence is independent of mesh parameters.
\begin{figure}[h]
    \centering
    \subfloat{{\includegraphics[height=3.5cm,width=6cm]{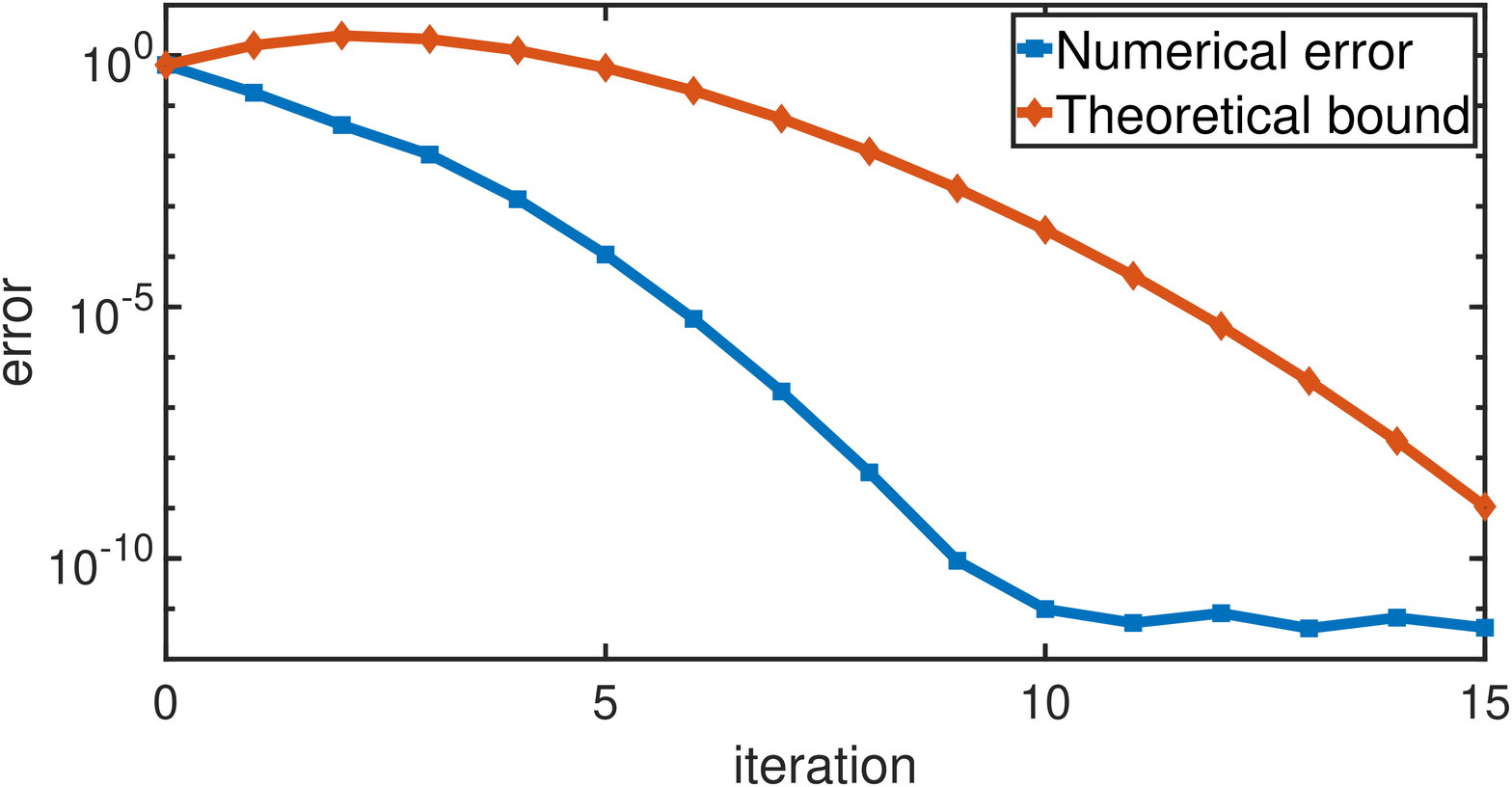} }}
     \subfloat{{\includegraphics[height=3.5cm,width=6cm]{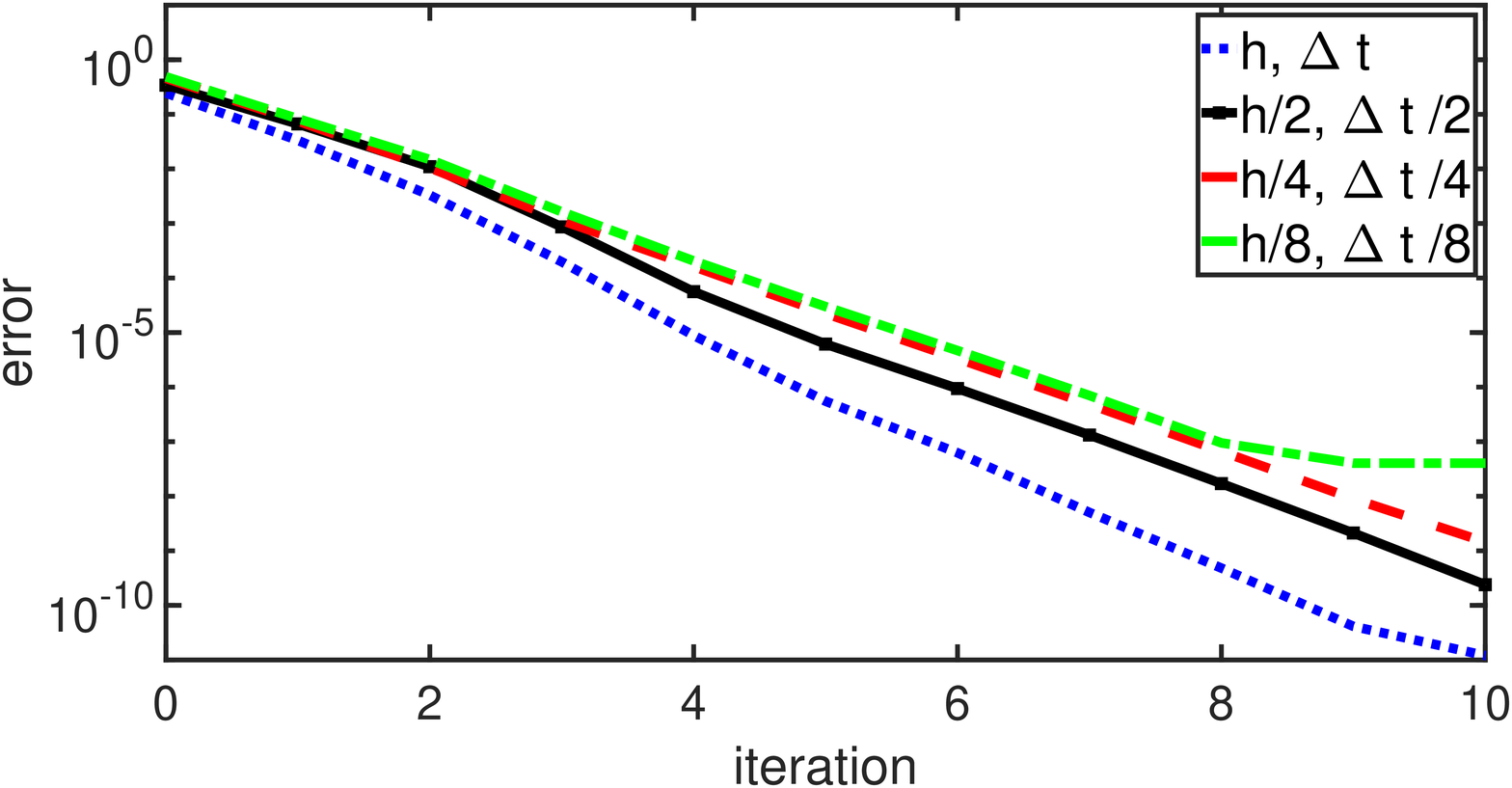} }}
    \caption{NPA-II: On the left comparison of theoretical and numerical error, and on the right convergence of refined solution.}
    \label{diff_errbound_npa2}
\end{figure}
We plot the error curves on the left in Figure \ref{diff_TN_npa2} for short as well as long time window with $\epsilon=0.0725, J=150$ and $h=1/64$. To see the dependency on the parameter $\epsilon$, we plot the error curves on the right in Figure \ref{diff_TN_npa2} for different $\epsilon$ by taking $T=1, N=50, J=200$. We can see that the NPA-II is independent of the choice of $\epsilon$. At this point we can compare NPA-I and NPA-II as both have the fine solution given by \eqref{approx1}. Between these two, NPA-II is expansive because of the nonlinear coarse solver and we take almost same number iteration to converge as in the case of NPA-I. Therefore it is better to use NPA-I while computing the nonlinear approximation of the CH equation. We skip the numerical experiments in 2D as we observe similar behaviour as in 1D.
\begin{figure}
    \centering
    \subfloat{{\includegraphics[height=3.5cm,width=6cm]{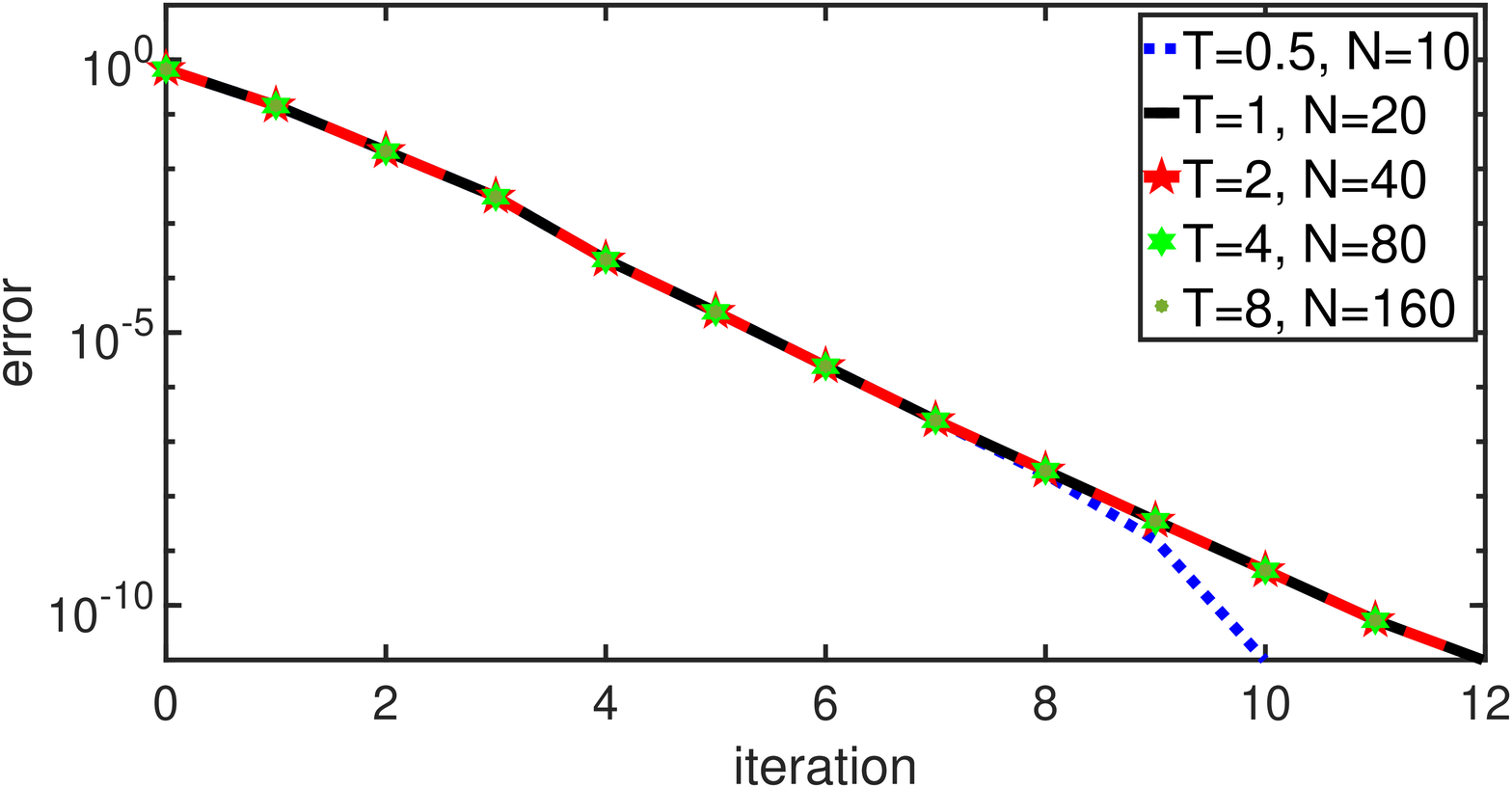} }}
     \subfloat{{\includegraphics[height=3.5cm,width=6cm]{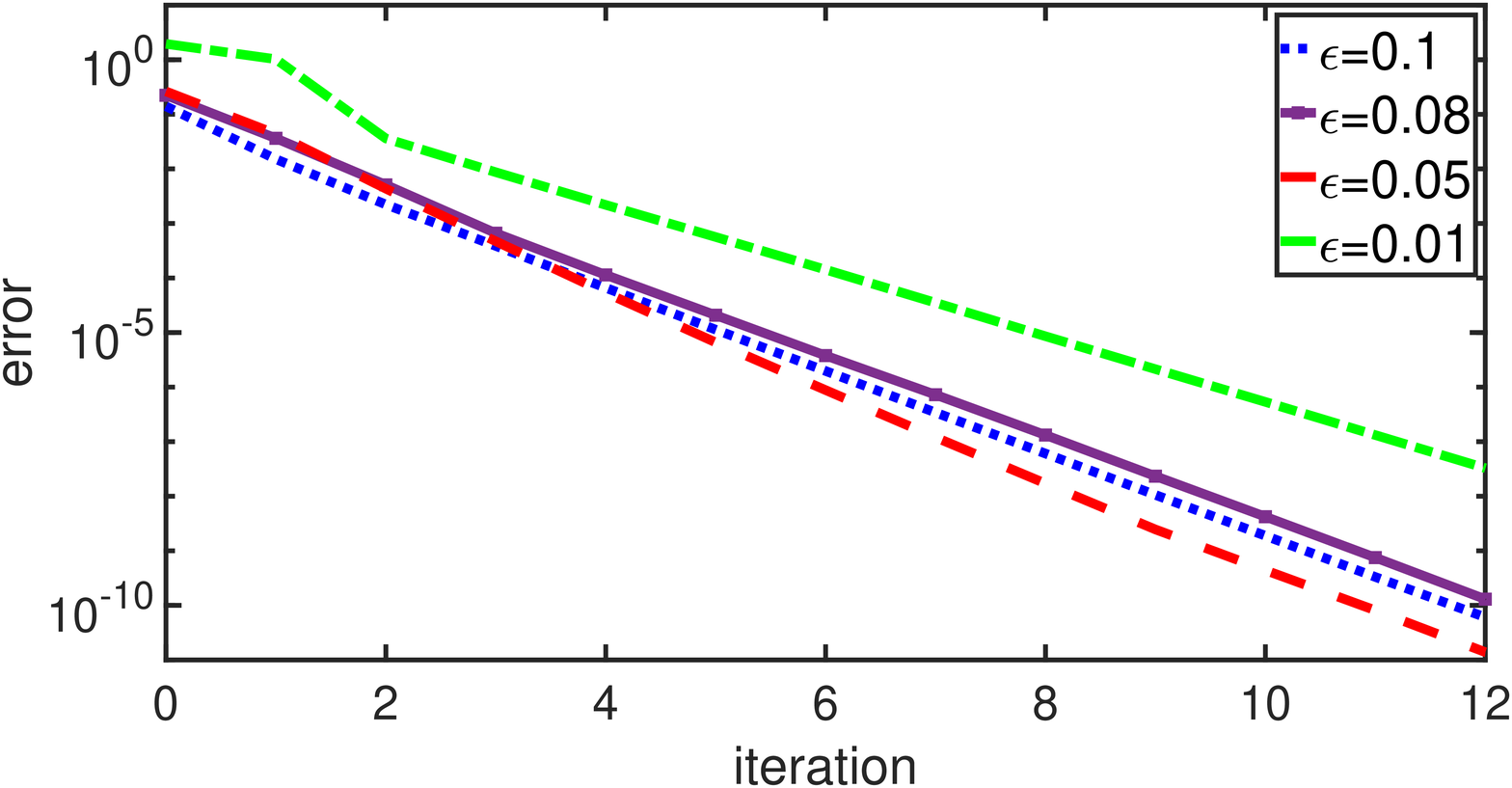} }}
    \caption{NPA-II: On the left convergence for different $T, N$, and on the right $\epsilon$ dependency on convergence.}
    \label{diff_TN_npa2}
\end{figure}